\documentclass[review]{siamart1116}

\usepackage{lipsum}
\usepackage{amsfonts}
\usepackage{graphicx}

\usepackage{algorithmic}

\usepackage{amssymb}
\usepackage{amsmath}
\usepackage{amsfonts,url}
\usepackage{tikz}
\usepackage{pgf}
\usepackage[top=1in, bottom=1in, left=1.5in, right=1.5in]{geometry}
\usetikzlibrary{positioning,chains,fit,shapes,calc,arrows,petri,topaths}
\usepackage{fancybox}
\usepackage{fontenc}
\usepackage[utf8]{inputenc}

\usepackage{subcaption}
\usepackage{float}

\usepackage{epstopdf}
\usepackage{epsfig}

\usepackage {tikz}
\usetikzlibrary {positioning}
\definecolor {processblue}{cmyk}{0.96,0,0,0}

\ifpdf
\DeclareGraphicsExtensions{.eps,.pdf,.png,.jpg}
\else
\DeclareGraphicsExtensions{.eps}
\fi

\numberwithin{theorem}{section}

\newcommand{\TheTitle}{Relative Optimality Conditions and Algorithms for \\ Treespace Fr\'{e}chet Means} 
\newcommand{\RunningTitle}{Fr\'{e}chet Mean Trees: Optimality \& Algs.} 
\newcommand{\TheAuthors}{S. Skwerer, S Provan, J.S. Marron}

\headers{\RunningTitle}{\TheAuthors}

\title{{\TheTitle}\thanks{Submitted to the editors November 2015.}}

\author{
	Sean Skwerer\thanks{
		(\email{seanskwerer@gmail.com} ).}
	\and
	Scott Provan\thanks{Department of Statistics and Operations Research University of North Carolina Chapel Hill (\email{provan@unc.edu}.}
	\and
	J.S . Marron \thanks{Department of Statistics and Operations Research University of North Carolina Chapel Hill (\email{marron@unc.edu}.}
}

\usepackage{amsopn}

\theoremstyle{plain}
\newtheorem{thm}{Theorem}[section]
\newtheorem{cor}[thm]{Corollary}
\newtheorem{lem}[thm]{Lemma}

\theoremstyle{definition}
\newtheorem{defn}[thm]{Definition}

\newcommand{\T}{\mathcal{T}}

\newcommand{\Or}{\mathcal{O}}
\newcommand{\A}{\mathcal{A}}
\newcommand{\B}{\mathcal{B}}
\newcommand{\V}{\mathcal{V}}

\newcommand{\tr}[1]{ \begin{list}{}{\setlength{\leftmargin}{#1em}} \item}
	
	\newcommand{\tl}{ \end{list}}

\newcommand{\mymin}[2]{ 
	\begin{array}{c l}
		\textrm{min} & #2 \\
		#1
	\end{array}
}

\providecommand{\abs}[1]{\lvert#1\rvert}
\providecommand{\norm}[1]{\lVert#1\rVert}

\ifpdf
\hypersetup{
  pdftitle={\TheTitle},
  pdfauthor={\TheAuthors}
}
\fi

\begin{document}

\maketitle

\begin{abstract}
Recent interest in treespaces as well-founded mathematical domains for phylogenetic inference and statistical analysis for populations of anatomical trees has motivated research into
efficient and rigorous methods for optimization problems on treespaces.
A central problem in this area is computing an average of phylogenetic trees, which is equivalently characterized as the minimizer of the Fr\'{e}chet function.
The Fr\'{e}chet mean can be used for statistical inference and exploratory data analysis: for example it can be leveraged as a test statistic to compare groups via permutation tests, or to find trends in data over time via kernel smoothing.
By analyzing the differential properties of the Fr\'{e}chet function along geodesics in treespace we obtained a theorem describing a decomposition of the derivative along a geodesic.
This decomposition theorem is used to formulate optimality conditions
which are used as a logical basis for an algorithm to verify relative optimality at points where the Fr\'{e}chet function gradient does not exist.
\end{abstract}



\section{Introduction}\label{sec:intro}

The space of phylogenetic trees introduced by Billera, Holmes and Vogtman \cite{BHV} is a metric space in which each point corresponds to a hypothetical evolutionary history. 
This space will be referred to as BHV Treespace. 
Inference about the evolutionary relationships of species
	has been a long standing issue in biology.
	As such phylogenetics is a mature field, with many approaches. 
	For book length treatments of phylogenetic inference and related mathematics see \cite{Felsenstein2004,gascuel2005mathematics}. 
	A space of phylogenetic trees is a pivotal mathematical concept for sound inference of evolutionary relationships \cite{holmes2003statistics,holmes2005statistical}.
	Although the specific metric for a space phylogenetic trees, or scope of a space of models for evolutionary relationships to include various types of trees or other structures such as networks are issues which may not yet be completely resolved, any such space will have a non-Euclidean geometry.
	The solution of the Fr\'{e}chet mean problem on BHV Treespace is a relevant research problem for phylogenetics and for applications involving modeling biological forms as trees. Treespace has also been used in statistical analyses where populations of lungs \cite{feragen2012hierarchical} and arteries \cite{Skwerer2014} are modeled as trees.
The central research problem of this paper is efficient computation of the Fr\'{e}chet mean of a discrete sample of points in treespace.
More specifically this paper focuses on deriving optimality conditions through the analysis of derivatives along geodesic paths issuing from singular points in treespace and an iterative interior point method for local optimization.

The main challenge here is that the Fr\'{e}chet function gradient is not well-defined everywhere in treespace. 
	The stratified structure of treespace is formed
	by gluing together Euclidean orthants (this is described in detail in Sec. \ref{sec:TreeSpace}).
	Gradients are not well defined at such glued points due to this non-differentiable geometry.
	Even if the domain of the Fr\'{e}chet optimization is restricted to a single orthant of treespace while the trees $T_1,...,T_n$ are supported on the entire space, the value of the Fr\'{e}chet function is piecewise continuous, but not differentiable.
	To be precise, on the boundary of the orthant, in directions from boundary to interior, directional derivatives of the Fr\'{e}chet function exist but the gradient is not well defined; and the interior of an orthant of treespace can be subdivided into regions where the Fr\'{e}chet function is $C^\infty$ differentiable on their interiors but it is only $C^1$ differentiable on their boundaries.

Derivative free search procedures can be used to avoid these non-differentiability issues (discussed further in Sec. \ref{sec:discussion}).
These derivative free procedures are designed to converge in distance to the optimal point. When the Fr\'{e}chet mean is not a completely resolved tree, these procedures, which in practice can continue only for finite number of iterations, will return trees having edges that are not present in the Fr\'{e}chet mean.

When the Fr\'{e}chet mean is a fully resolved  tree, the gradient of the Fr\'{e}chet function at that point in treespace exists, and its optimality can be confirmed by checking the gradient is zero.
However, to date, there are no known quickly verifiable certificates for a treespace Fr\'{e}chet mean in general.
Our main results in this paper are (1) a decomposition theorem for derivatives along geodesics and (2) an algorithm to find the minimizer of the Fr\'{e}chet function when the domain is restricted to a single orthant of treespace. 
This algorithm will determine the optimal point in a single orthant of treespace even when the optimal point represents a tree which is not fully resolved.

Essentially, our new algorithm minimizes the directional derivative of the Fr\'{e}chet function over the set of all directions issuing from a point to within the closure of an orthant, and verifies that there is no direction where the derivative is negative.
	The crux of this approach is that when the Fr\'{e}chet function gradient does not exist then generally the gradient of the directional derivative function does not exist.
	This issue is solved with a nested optimality condition based on recursive decomposition of the directional derivatives.

The remaining contents of this paper are organized as follows. Background about treespace and geodesics is presented in Sec. \ref{sec:background}. 
In Sec. \ref{sec:discussion} we give a high-level discussion of technical issues for the Fr\'{e}chet optimization problem.
The main theorems are discussed in Sec. \ref{sec:main_results} (proofs are in Sec. \ref{sec:DiffAnalysis}).
In Sec. \ref{sec:IntPointMethods} a method finding the minimizer
of the Fr\'{e}chet function in a fixed orthant of treespace is presented. 
The focus of Sec. \ref{sec:DiffAnalysis} is our analysis of the differential properties of the Fr\'{e}chet function.
Concluding remarks and further research directions are in Sec. \ref{sec:conclusion}.

\section{Background}\label{sec:background}
Definitions and descriptions of phylogenetic trees and Billera, Holmes, and Vogtman (BHV) Treespace are given in Sec. \ref{sec:PhylogeneticTrees} and Sec. \ref{sec:TreeSpace}.
Treespace geodesics are defined in Sec. \ref{sec:Geodesics}.
In Sec. \ref{sec:FMproblem} we state the Fr\'{e}chet mean optimization problem. 
In Sec. \ref{sec:VistalCells} we describe how the combinatorics of treespace 
geodesics lead to subdivision of treespace
into regions where the Fr\'{e}chet function has a fixed algebraic form.


\subsection{Phylogenetic trees}\label{sec:PhylogeneticTrees}

Evolutionary histories or hierarchical relationships are often represented graphically as phylogenetic trees.  In biology, the evolutionary history of species or operational taxonomic units (OTU's) is represented by a tree. 
The root of the tree corresponds to a common ancestor. Branches indicate speciation of a nearest common ancestor into two or more distinct taxa. The leaves of the tree correspond to the present species whose history is depicted by the tree. An overview of the data used in phylogenetic inference is presented in Sec. \ref{sec:GenData} and mathematical definitions for phylogenetic trees are given in Sec. \ref{sec:PhyloTreeDefns}.

	\subsubsection{Genetic Data for Phylogenetics}\label{sec:GenData}
	Phylogenetic trees and inferences about the evolutionary relationships of species are typically made from
	molecular sequence data of DNA, RNA, amino acids, or proteins, aligned for interspecies comparison. Table \ref{tab:GenSeqData} contains a small artificial example of DNA base pair data.
	
	\begin{table}[h!]
		\centering
		\caption{Artificial example of nine aligned DNA base pairs from seven species of snake.}
		\label{tab:GenSeqData}
		\begin{tabular}{l||c|c|c|c|c|c|c|c|c|c}
			Species & a& b& c & d & e & f & g & h & i \\
			\hline
			1. King Cobra & A &  A & T & A & C & T & A & A & C \\
			2. Copper Head & A & A & T & A & G & T & T & A & G\\
			3. Black Mamba & A & A & T & G & G & C & T & A & G \\
			4. Corn Snake & A & C & T & G & G & C & A & A & G \\
			5. Boa Constrictor & T & C & T & A & C & T & A & A & G \\
			6. Coral Snake & T & C & A & G & C & T & A & A & G\\
			7. Cotton Mouth & T & C & A & G & C & T &  A & T & C\\
		\end{tabular}
	\end{table}
	
	In this example, the pattern in each column yields a partition of the species into two groups.
	Column a partitions the snakes into \{King Cobra,  Copper Head, Black Mamba, Corn Snake\} and \{Boa Constrictor, Coral Snake, Cotton Mouth\}.
	Snakes with A (adenine) in column b are a subset of the group of snakes with A in column a, and likewise snakes with T (thymine) in column a are a subset of the group of snakes with C (cytosine) in column b.
	Together column a and column b partition into three groups: \{King Cobra,Copper Head, Black Mamba\}, \{Corn Snake\}, and \{Boa Constrictor, Coral Snake, Cotton Mouth\}.
	Column a and column b are an example of compatible splits (see. Sec. \ref{sec:PhyloTreeDefns} for definition for compatible splits).
	On the other hand, column b and column d are an example of incompatible splits.
	The pattern in column b separates Black Mamba from Corn Snake, while the pattern in column d puts these snakes in the same group.
	
	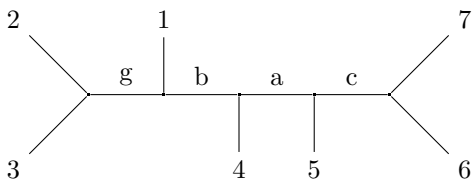
\begin{figure}
		\begin{center}
			
			\begin{tikzpicture}
			\node[] (1) at (0,1) {2};
			\node[] (2) at (0,-1) {3};
			\node[fill,inner sep=0pt,minimum size=1pt] (3) at (1,0) {};
			\node[fill,inner sep=0pt,minimum size=1pt] (4) at (2,0) {};
			\node[] (5) at (2,1) {1};
			\node[fill,inner sep=0pt,minimum size=1pt] (6) at (3,0) {};
			\node[] (7) at (3,-1) {4};
			\node[fill,inner sep=0pt,minimum size=1pt] (8) at (4,0) {};
			\node[] (9) at (4,-1) {5};
			\node[fill,inner sep=0pt,minimum size=1pt] (10) at (5,0) {};
			\node[] (11) at (6,-1) {6};
			\node[] (12) at (6,1) {7};
			
			\path[-] (1) edge (3);
			\path[-] (2) edge (3);
			\path[-] (3) edge node[above]{g} (4);
			\path[-] (4) edge (5);
			\path[-] (4) edge node[above]{b} (6);
			\path[-] (6) edge (7);
			\path[-] (6) edge node[above]{a} (8);
			\path[-] (8) edge (9);
			\path[-] (8) edge node[above]{c} (10);
			\path[-] (10) edge (11);
			\path[-] (10) edge (12);
			
			\end{tikzpicture}
			
		\end{center}
		
		\caption{A phylogenetic tree for seven species of snake from data in Table \ref{tab:GenSeqData}. This tree represents the partitions of the species into groups from columns a, b, c, and g. }
		\label{fig:DNASeqExTree}
	\end{figure}
	
	A set of compatible splits can be combined into a tree.
	The information in columns a, b, c, g and h suggests one possible configuration of evolutinary relationships for these snake species, which is depicted as a tree in Figure \ref{fig:DNASeqExTree}.

	The split from column d is incompatible with columns a, b, c, and g, and such conflicting information cannot be put into a single tree.
	Genetic data may support multiple evolutionary histories because genes, the elementary units of inheritance \cite{gerstein2007gene}, may evolve independently from each other as they are driven by independent environmental factors.
	The problem of forming a species tree from gene trees is known as gene tree species tree reconciliation \cite{larget2010bucky,maddison1997gene}.

\subsubsection{Mathematical Definitions for Phylogenetic Trees}\label{sec:PhyloTreeDefns}
A \emph{labeled tree} is a tree $T$ with $r+1$ leaves distinctly labeled using the index set $I =\{0,1,\ldots,r\}$.
An edge $e$ is characterized by a \emph{split}, which is a partition of $I$ into two disjoint sets of labels, $X_e$ and $\bar{X}_e$. Edge $e$ is present in tree $T$ if and only if deleting $e$ from $T$ yields one subtree with leaves distinctly labeled by $X_e$ and another subtree with leaves distinctly labeled by $\bar{X}_e$.

A \emph{phylogenetic tree} is a labeled tree with weighted edges: 
the set of edges for a tree $T$ is written $E_T$, and edge weights are a function from the edges of $T$ to the positive real numbers, $|\;\;|_T:E_T \to \mathbb{R}_{>0}$.
The lengths of edges in a tree will typically represent some measure of genetic difference, or in some models the passage of time. 
	When edge lengths represent the passage of time, the leafs of the tree are associated with contemporary species, while the root is associated with an ancestral species. 
	In that case the edge lengths must be normalized so that the passage of time represented by the path from the ancestor to each contemporary species is the same.
The \emph{topology of a phylogenetic tree} is the underlying graph and leaf labels separated from the edge lengths.
The topology of a phylogenetic tree is uniquely represented by the set of splits associated with its edges.
Formally, two splits $X_e \cup \bar{X}_e$ and $X_f\cup \bar{X}_f$ 
are \emph{compatible} if and only if $X_e \subset X_f$ and $\bar{X}_f \subset \bar{X}_e$, or $X_f \subset X_e$ and $\bar{X}_e \subset \bar{X}_f$. 
Compatibility can be interpreted in terms of subtrees: the subtree with leaves in bijection with $\bar{X}_e$
contains the subtree with leaves in bijection with $\bar{X}_f$, or vice versa.
The compatibility of splits on snake species from the aligned DNA base pairs in \ref{tab:GenSeqData} depicted as a graph in Figure \ref{fig:GenSeqLink}, with one node for each split and an edge connecting each compatible pair.
If every pair of splits in a set of splits is compatible then that set is said to be a compatible set.
Each distinct set of compatible splits is equivalent to a unique phylogenetic tree topology. A maximal tree topology is one in which no additional interior edges can be introduced i.e. $|E_T|=2r-1$, or equivalently every interior vertex has degree 3.

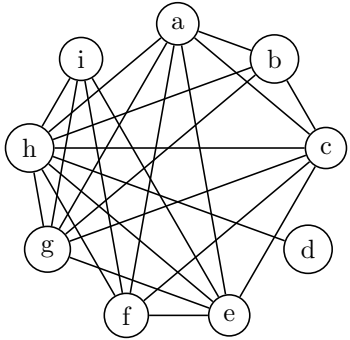
\begin {figure}
\begin{center}
	
	\begin {tikzpicture}[-latex ,auto ,node distance =4 cm and 5cm ,on grid ,
	semithick ,
	state/.style ={ circle ,top color =white , bottom color = white ,
		draw,black , text=black , minimum width =1 cm}]
	
	\foreach [count=\i] \j in {a,b,...,i}{
		\draw ( 130-\i*360/9 : 2cm) node[state,minimum size=0.5cm](\i){\j};
	}
	
	\path[black] (1) edge[-] (2);
	\path[black] (1) edge[-] (3);
	\path[black] (1) edge[-] (5);
	\path[black] (1) edge[-] (6);
	\path[black] (1) edge[-] (7);
	\path[black] (1) edge[-] (8);

	\path[black] (2) edge[-] (3);
	\path[black] (2) edge[-] (7);
	\path[black] (2) edge[-] (8);

	\path[black] (3) edge[-] (5);
	\path[black] (3) edge[-] (6);
	\path[black] (3) edge[-] (7);
	\path[black] (3) edge[-] (8);

	\path[black] (4) edge[-] (8);

	\path[black] (5) edge[-] (6);
	\path[black] (5) edge[-] (7);
	\path[black] (5) edge[-] (8);
	\path[black] (5) edge[-] (9);

	\path[black] (6) edge[-] (8);
	\path[black] (6) edge[-] (9);

	\path[black] (7) edge[-] (8);
	\path[black] (7) edge[-] (9);

	\path[black] (8) edge[-] (9);
	
\end{tikzpicture}

\end{center}

\caption{Compatibility graph for splits on snake species from aligned DNA base pairs in Table \ref{tab:GenSeqData}.}
\label{fig:GenSeqLink}
\end{figure}

\subsection{Construction of BHV Treespaces}\label{sec:TreeSpace}
A BHV Treespace, $\T_r$ is a geometric space in which each point represents a phylogenetic tree having  leaves in bijection with a fixed label set $\{0,1,2,\ldots,r\}$.

A \emph{non-negative orthant} is a copy of the subset of $n$-dimensional Euclidean space defined by restricting each coordinate to non-negative values, $\mathbb{R}^{n}_{\geq 0}$. Here, only non-negative orthants are used, so we use orthant to mean non-negative orthant.
An \emph{open orthant}
is the set of positive points in an orthant.
Phylogenetic treespace is a union of many orthants, each corresponding to a distinct tree topology,
wherein the coordinates of a point are interpreted as the lengths of edges.
For a given set of compatible edges $E$, the associated orthant is denoted $\Or(E)$, 
and for a given tree $T$, the minimal orthant in treespace containing
that point is denoted $\Or(T)$.
Trees in $\mathcal{T}_r$ have at most $r-2$ interior edges.
Each orthant of dimension $r-2$ corresponds to a combination of $r-2$ compatible edges. 
Orthants are glued together along common axes.
The shared faces of facets with $k$ positive coordinates are called the $k$-dimensional faces of 
treespace.

Take $\T_4$ as an example.
There are ten possible splits (not including five leaf edges), see Table \ref{table:T4Splits} for a list.
These splits can be combined into fifteen distinct tree topologies, comprised of pairs of compatible splits.
Each compatible pair is associated with a copy of $\mathbb{R}^2_{\geq 0}$,
one axis of the orthant for each edge in the pair.  
These fifteen orthants are glued together along common axes.
Views of two parts of $\T_4$ are displayed in Figure \ref{T_4_parts}.
See Figure \ref{T_4_link} for a visualization of the split-split compatibility graph of $\T_4$ \footnote{An interesting fact is that the compatibility graph of $\T_4$ is a Peterson graph. }.
\begin{table}
\begin{centering}
\begin{tabular}{|c|c|c|c|c|}
	\hline
	$\{0,1\}|\{2,3,4\}$ & $\{0,1,2\}|\{3,4\}$ & $\{0,2\}|\{1,3,4\}$ & $\{0,2,4\}|\{1,3\}$ & $\{0,1,3\}|\{2,4\}$\\
	\hline
	$\{0,1,4\}|\{2,3\}$ & $\{0,2,3\}|\{1,4\}$ & $\{0,3\}|\{1,2,4\}$ & $\{0,3,4\}|\{1,2\}$ & $\{0,4\}|\{1,2,3\}$\\
	\hline
\end{tabular}
\caption{Ten partitions of $\{0,1,2,3,4\}$ which are splits for internal edges in the BHV Treespace $\T_4$. Comptable splits are combined to make tree topologies as depicted if Fig. \ref{fig:T_4_parts}.}
\label{table:T4Splits}
\end{centering}
\end{table}

\begin{figure}[H]
\centering
\begin{subfigure}[b]{\textwidth}
\centering
\includegraphics[width = 0.6\textwidth]{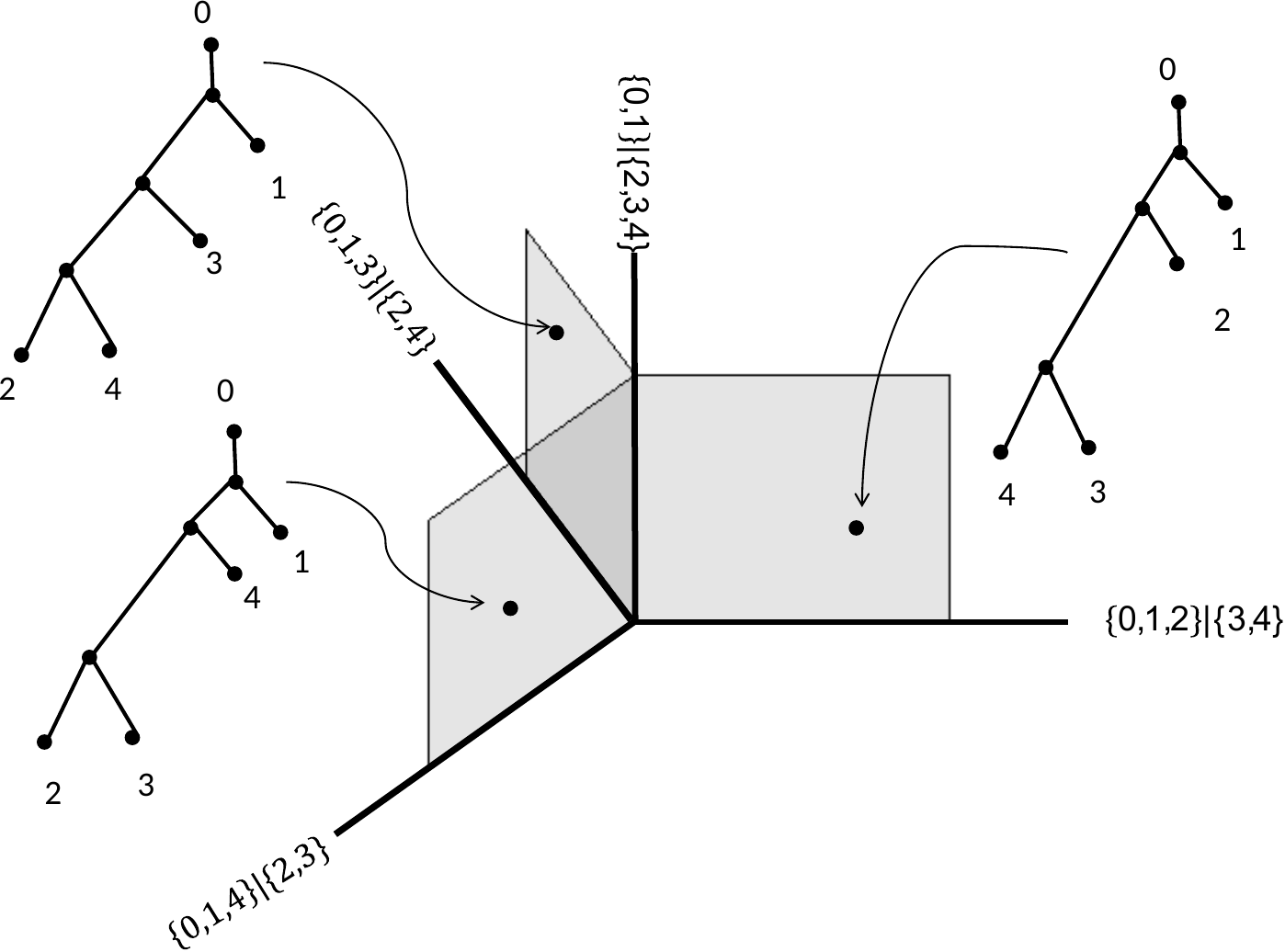}
\caption{A half-open book with three pages.  In this diagram the \emph{pages} of the book are three copies of $\mathbb{R}_{\geq 0}^2$ and the \emph{spine} is a copy of $\mathbb{R}_{\geq 0}$. The spine is labeled with the split $\{0,1\}|\{2,3,4\}$. Each page has the spine as one axis and the other axis is labeled with a split compatible with $\{0,1\}|\{2,3,4\}$. In $\T_4$ every one of the ten splits of $\{0,1,2,3,4\}$ is the label for the spine of a half-open book. }
\label{open_book}
\end{subfigure}
\begin{subfigure}[b]{\textwidth}
\centering
\includegraphics[width = 0.6\textwidth]{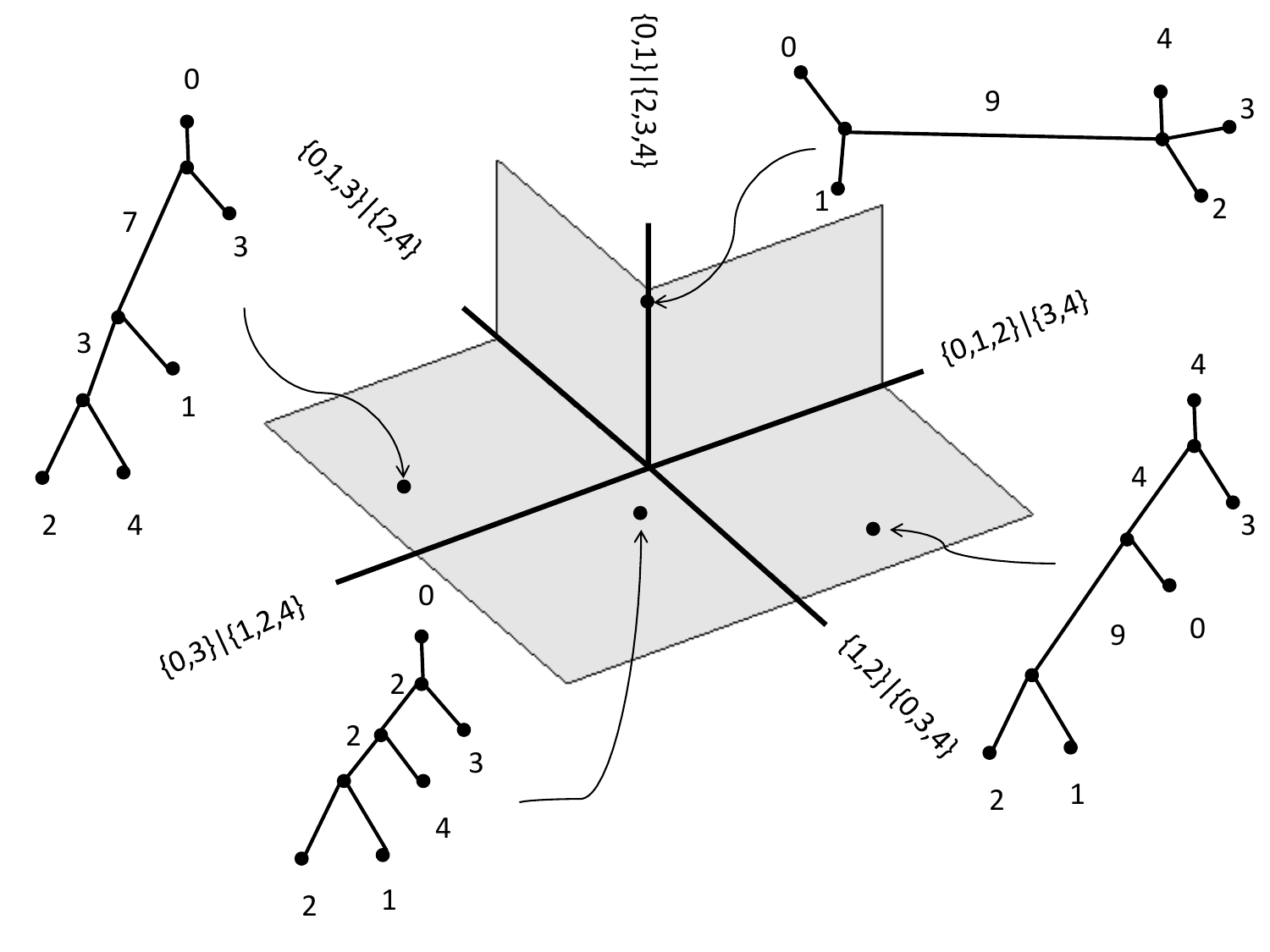}
\caption{A five-cycle. A \emph{five-cycle} is five copies of $\mathbb{R}^2_{\geq 0}$ glued together along commonly labeled axes. }
\label{five_cycle}
\end{subfigure}
\caption{}
\label{fig:T_4_parts}
\end{figure}

Each clique in the split-split compatibility graph represents a compatible combination of splits, or equivalently the topology of a phylogenetic tree. A graph is \emph{complete} if there is an edge between every pair of vertices. In a graph, a \emph{clique} is a complete subgraph. Each full phylogenetic tree is a maximal clique in the split-split compatibility graph because a clique represents a set of mutually compatible splits. The split-split compatibility graph of $\T_4$ has fifteen maximal cliques, each of which is represented an edge in the graph. The split-split compatibility graph of $\T_4$ determines how the orthants of $\T_4$ are connected.

\begin{figure}[H]
\centering
\includegraphics[width = 0.5\textwidth]{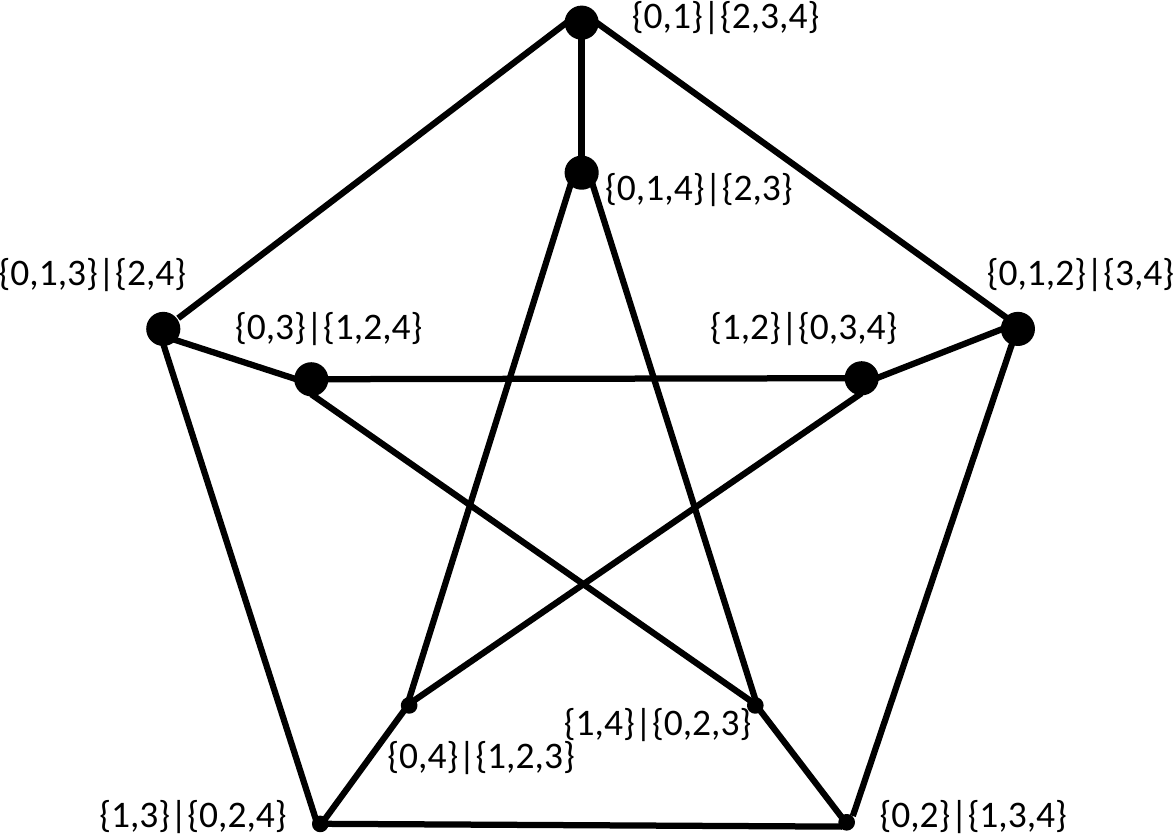}
\caption{Split-split compatibility graph of ${\mathcal T}_4$. Each split has a node. Two splits are compatible if they are connected by an arc. This shows the overall connectivity of $T_4$, all possible splits for $\{0,1,2,3,4\}$, and all possible topologies for 4-trees. Each vertex and the three edges emanating from it in the graph corresponds to a copy of an open book like in Figure \ref{open_book}. Each five-cycle in the graph is a copy of a five-cycle depicted in Figure \ref{five_cycle}}
\label{T_4_link}
\end{figure}

The compatibilty of splits from genetic sequence data in Table \ref{tab:GenSeqData} is given in the link graph in Figure \ref{fig:GenSeqLink}.

\subsection{BHV Treespace geodesics}\label{sec:Geodesics}
We now give an explicit description of geodesics in treespace.
Let $X \in \T_r$ be a variable point and let $T \in \T_r$ be a fixed point.
Let $\Gamma_{XT}= \{\gamma(\lambda)|0\leq \lambda \leq 1\}$ be the geodesic path from $X$ to $T$.
Let $C$ be the set of edges which are compatible in both trees, that is the union of the largest subset of $E_X$ which is compatible with every edge in $T$ and the largest subset of $E_T$ which is compatible with every edge in $X$. 

The following notation for the Euclidean norm of the lengths of a set of edges $A$ in a tree $T$ will be used frequently,
\begin{equation}
||A||_T = \sqrt{ \sum_{e \in A}{|e|_T^2}}
\end{equation}
or without the subscript when it is clear to which tree the lengths are from.

A support sequence is a pair of disjoint partitions, $A_1\cup \ldots \cup A_k =E_X\setminus C$ and $B_1\cup\ldots\cup B_k=E_T \setminus C$.
\begin{thm}\cite{OwenProvan}
A support sequence $(\A,\B)=(A_1,B_1),\ldots,(A_k,B_k)$ corresponds to a geodesic if and only if it satisfies the following three properties:
\begin{itemize}
\item[\rm (P1)] For each $i>j$, $A_i$ and $B_j$ are compatible
\item[\rm (P2)] $\frac{\norm{A_1}}{\norm{B_1}} \leq \frac{\norm{A_2}}{\norm{B_2}} \leq \ldots \leq \frac{\norm{A_{k}}}{\norm{B_{k}}}$
\item[\rm (P3)] For each support pair $(A_i, B_i)$, there is no nontrivial partition $C_1 \cup C_2$ of $A_i$, and partition $D_1 \cup D_2$ of $B_i$, such that $C_2$ is compatible with $D_1$ and $ \frac{\norm{C_1}}{\norm{D_1}} < \frac{\norm{C_2}}{\norm{D_2}}$
\end{itemize}
The geodesic between $X$ and $T$ can be represented in ${\mathcal T}_r$ with legs
\begin{displaymath}
\Gamma_l=\left\{\begin{array}{ll}
\left[\gamma(\lambda):\; \frac{\lambda}{1-\lambda}\leq\frac{\norm{A_1}}{\norm{B_1}}\right],
&l=0\\[.7em]
\left[\gamma(\lambda):\; \frac{\norm{A_i}}{\norm{B_i}}\leq\frac{\lambda}{1-\lambda}\leq\frac{\norm{A_{i+1}}}{\norm{B_{i+1}}}\right],
&l=1,\ldots,k-1,\\[.7em]
\left[\gamma(\lambda):\; \frac{\lambda}{1-\lambda}\geq\frac{\norm{A_k}}{\norm{B_k}}\right],
&l=k\end{array}\right.
\end{displaymath}
The points on each leg $\Gamma_l$ are associated with tree $T_l$ having edge set

\begin{displaymath}
\begin{array}{rcl}
E_l&=&B_1\cup\ldots\cup B_l\cup A_{l+1}\cup\ldots\cup A_k\cup C
\end{array}
\end{displaymath}

Lengths of edges in $\gamma(\lambda)$ are

\begin{displaymath}
|e|_{\gamma(\lambda)}=\displaystyle\left\{\begin{array}{ll}
\frac{(1-\lambda)\norm{A_j}-\lambda \norm{B_j}}{\norm{A_j}}|e|_X&e\in A_j\\[1em]
\frac{\lambda \norm{B_j}-(1-\lambda)\norm{A_j}}{\norm{B_j}}|e|_{T}&e\in B_j\\[1.5em]
(1-\lambda)|e|_X+\lambda |e|_{T}&e\in C\\
\end{array}.\right.
\end{displaymath}

The length of $\Gamma$ is
\begin{equation}\label{pathlength}
d(X,T) = \sqrt{\sum_{l=1}^k{ (\norm{A_l}+\norm{B_l})^2 + \sum_{e \in C}{ (|e|_X-|e|_T)^2}     }}
\end{equation}
and we call this the geodesic distance from $X$ to $T$.
\end{thm}

\subsection{Fr\'{e}chet means in BHV Treespace}\label{sec:FMproblem}
For a given data set of $n$ phylogenetic trees in $\T_r$, $T^1, T^2,\ldots,T^n$, 
the \emph{Fr\'{e}chet function} is the sum of squares of geodesic distances from the data trees to a variable tree $X$. 
A geodesic $\gamma:[0,1]\to \T_r$ is the shortest path between its endpoints.
The geodesic from $X$ to $T^i$ is characterized by a geodesic support, $(\A^i,\B^i) = \left ( (A^i_1,B^i_1),\ldots, (A^i_{k^i},B^i_{k^i}) \right )$ \cite{OwenProvan}.
Given the geodesic supports  $(\A^1,\B^1),\ldots, (\A^n,\B^n)$ the Fr\'{e}chet function is
\begin{equation}\label{FrechetFunction}
F(X) = \sum_{i = 1}^n d(X,T^i)^2=\sum_{i = 1}^n \left ( \sum_{l = 1}^{k^i} (\norm{A_l^i}+\norm{B_l^i})^2 + \sum_{e \in C^i} (|e|_X - |e|_{T^i})^2 \right)
\end{equation}

The objective is to solve the Fr\'{e}chet optimization problem
\begin{equation}\label{FrechetOptimization}
\mymin{X \in \T_r}{F(X)}
\end{equation}
Elementary Fr\'{e}chet function properties:
\begin{itemize}
\item The Fr\'{e}chet function is continuous because the geodesic distances $d(X,T^i)$ are continuous \cite{OwenProvan,Miller2015}.
\item The Fr\'{e}chet function $F(X)$ is strictly convex \cite{Sturm}, that is $F\circ \gamma:[0,1] \to \mathbb{R}$ is strictly convex for every geodesic $\gamma(\lambda)$ in $\T_r$. 
\end{itemize}
As a consequence of these properties we have the following result.
\begin{lem}
The Fr\'{e}chet mean exists and is unique.
\end{lem}
\begin{proof}
A strictly convex function either has a unique minimizer or can be made
arbitrarily low. 
Assuming that the data points are finite, 
then a minimizer of the Fr\'{e}chet function must also be finite.
Therefore the Fr\'{e}chet function has a unique minimizer.  
\end{proof}
\begin{defn}
The Fr\'{e}chet mean, $\bar{T}$, is the unique minimizer of the Fr\'{e}chet function. 
\end{defn}

\subsection{Vistal subdivision of treespace}\label{sec:VistalCells}

The  value of 
the Fr\'{e}chet function at $X$
depends on the geodesics from $X$ to each of the data trees.
Treespace can be subdivided
into regions where the combinatorial form of geodesics from $X$ to the data
trees are all fixed.
Given a source tree $T$, a vistal facet is a region of treespace, $\V$, where a fixed support is valid
for the geodesic from any tree $X$ in $\V$ to $T$.
\begin{defn}
\cite[Def. 3.3]{Miller2015} Let $T$ be a tree in $\T_r$. Let $\Or$
be a maximal orthant containing $T$. The \emph{previstal facet}, ${\cal V}(T,\Or ; \A, \B)$, is the set of variable trees, $X$ $\in \Or$, for which the geodesic joining $X$ to $T$ has support $(\A,\B)$ satisfying $(P2)$ and
$(P3)$ with strict inequalities.
\end{defn}
\noindent A pre-multi-vistal facet is an intersection of previstal facets of $T^1,\ldots,T^n$.
Pre-multi-vistal facets are regions where the Fr\'{e}chet function can be represented
with a fixed algebraic form.
Analysis of the differential properties of the Fr\'{e}chet function at points on shared faces of pre-multi-vistal facets is important
for Thm. \ref{thm:DirDerDecomposition}.
For deeper analysis of the geometry and combinatorics of vistal facets see \cite{Miller2015}.

\section{Problem Discussion}\label{sec:discussion}
The Fr\'{e}chet optimization problem, in BHV Treespace, requires both selecting the minimizing tree topology and specifying its edge lengths.
Tree topologies are discrete and so the problem of selecting the minimizing tree topology is a combinatorial optimization problem; 
however it is possible to make search strategies which take advantage of the continuity of BHV Treespace to find
the correct tree topology. 
It is natural to consider this problem in two modes of search: global i.e. strategies which change the topology and edge lengths; and local i.e. strategies which only adjust edge lengths. 
One motivation to consider global search and local search separately is that the local optimization problem
is convex optimization constrained to a Euclidean orthant.

The global search problem is challenging because 
the geometry of treespace creates difficulty 
in two essential parts of optimization (1) making progress towards optimality and (2)
verifying optimality. 
In treespace a metrically small neighborhood can actually be quite large in a certain sense. 
In constructing the space, the topological identification of the shared faces of orthants may create points in the closure of many orthants. 
In terms of trees, the neighborhood around a tree $X$, $N(X)$, is comprised
not only of trees with the same topology as $X$ but also
trees which have $X$ as a contraction. However, the list of tree topologies which have a particular tree $X$ as a contraction can be quite large. For example, if $X$ is a star tree then $X$ is a contraction of any tree i.e. $X$ is a contraction of $(2r-3)!!$ maximal phylogenetic tree topologies. 

Local optimality conditions for non-differentiable functions
are based on the rate of change of the objective function along directions issuing from
a point. 
Since the neighborhood of a point $X$ contains
all trees which have $X$ as a contraction verifying 
that $X$ is optimal requires demonstrating that
any tree which contains $X$ as a contraction
has a larger Fr\'{e}chet function value.
For example, when $X$ is a star tree, $N(X)$ contains every tree
with the same leaf edges as $X$, and having infinitesimal interior edge lengths. In this sense finding a descent direction
can be essentially as hard as finding the topology of the Fr\'{e}chet mean itself.
Although exhaustive search among all possible tree topologies
will eventually yield the Fr\'{e}chet mean, there are more practical approaches.

Proximal point algorithms, a broad class of algorithms, are globally convergent
not only for the Fr\'{e}chet optimization problem, but are globally 
convergent for any well defined lower-semicontinuous convex optimization problem 
on a globally non-positively curved metric space \cite{bacak2013proximal}.
A globally non-positively curved metric space has a unique shortest path, called a geodesic, between any pair of points.
This class of algorithms has nice theoretical properties and certain implementations of proximal point algorithms are practical for the Fr\'{e}chet optimization problem
on globally non-positively curved orthant spaces. Orthant spaces are generalizations of treespace where the link at the origin can be an arbitrary graph, rather than a graph encoding valid phylogenetic trees \cite[Sec. 6.3]{Miller2015}.

Proximal point algorithms are applicable to optimization problems on metric spaces. The general problem is minimizing a function $f$ on a metric space $M$ with distance function $d:M\times M \to \mathbb{R}$.
A proximal point algorithm solves a sequence of penalized optimization problems
of the form
\begin{align}
\begin{displaystyle}
P_k(f): \min_{x^{k}}{ \quad f(x^k)+\alpha_k d^2(x^{k-1},x^k)}
\end{displaystyle}
\end{align}
where $\alpha_k$ influences the proximity of a solution to the point $x^{k-1}$.
Some good references for proximal point algorithms are \cite{bacak2014convex,Bertsekas2011,Li2009,Rockafellar1976}.

Global methods for optimizing the Fr\'{e}chet function i.e. methods which 
can move from one orthant of treespace to another have been shown to converge \cite{Bacak,Sturm,Miller2015}.

Implementing a generic proximal point algorithm to minimize the Fr\'{e}chet function on treespace does not seem advantageous. In particular, given a non-optimal point $X^0$ finding a point $X$ such that $F(X) < F(X^0)$ is not made any easier
by penalizing  the objective function $F(X)+\alpha d^2(X^0,X)$.
Penalizing the objective function with $\alpha d(X^0)$ does not 
provide any additional structure for checking the neighborhood of $X^0$, $N(X^0)$, for a descent direction.

Split proximal point algorithms avoid directly tackling the complicated problem
of minimizing $F(X)$ by solving many much easier subproblems. 
For objective functions which can be expressed as a sum of functions, $f = f^1+\ldots +f^m$, a split proximal point algorithm alternates among penalized optimization problems for each function. 
Let $\{1,2,\ldots,m\}$ index the functions $f^1,\ldots,f^m$. 
A generic split proximal point algorithm is:
choose some sequence $i_1,i_2,\ldots$ where each 
term in the sequence is an element of $\{1,2,\ldots,m\}$ and sequentially solve
the split proximal point optimization problem:
\begin{align}
\begin{displaystyle}
P_k(f^{i_k}): \min_{x^{k}}{ \quad f^{i_k}(x^k)+\alpha_k d^2(x^{k-1},x^k)}
\end{displaystyle}
\end{align}
Different versions of split proximal point algorithms are based on
the choice of the sequence $i_1,i_2,\ldots$ and the choice of the sequence $\{\alpha_k\}$. Naturally, a split proximal point procedure can be applied to the Fr\'{e}chet optimization problem by separating the Fr\'{e}chet function into a sum of squared distance functions, $F(X) = d^2(X,T^1)+\ldots +d^2(X,T^n)$.
For the Fr\'{e}chet function the split proximal point optimization problem is
\begin{align}
\begin{displaystyle}
P_k(d^2(X^k,T^{i_k})): \min_{x^{k}}{ \quad d^2(X^k,T^{i_k})+\alpha_k d^2(X^{k-1},X^k)}
\end{displaystyle}
\end{align}
For the Fr\'{e}chet mean optimization problem on a globally non-positively curved metric space, the solution to a split proximal point optimization problem 
can be obtained easily in terms of geodesics.
The solution to  $P_k(d^2(X^k,T^{i_k}))$
must be on the geodesic between $X^{k-1}$ and $T^{i_k}$.
The term $d^2(X^k,T^{i_k})$ is the squared distance from the 
variable point to $T^{i_k}$ and the term $d^2(x^{k-1},x^k)$ is the squared
distance from the search point to $X^{k-1}$. Given
any point, there is at least one point
on the geodesic between $X^{k-1}$ and $T^{i_k}$ for
which the value of both terms is at least as small.
Since $X^k$ must be on this geodesic, the distance from $X^k$ to $X^{k-1}$
and the distance from $X^k$ to $T^{i_k}$ can be parameterized
in terms of  the proportion, $t:0 \leq t \leq 1$, along the geodesic from $X^{k-1}$ to $T^{i_k}$:  $d(X^k,T^{i_k})=(1-t)d(X^{k-1},T^{i_k})$ and $d(X^k,X^{k-1})=td(X^{k-1},T^{i_k})$.
Parameterizing $d^2(X^k,T^{i_k})+\alpha_k d^2(X^{k-1},X^k)$
in terms of $t$ makes $P_k(d^2(X^k,T^{i_k}))$ into a problem
of minimizing a quadratic function in $t$.
The optimal step length is $t = \frac{\alpha}{1+\alpha}$.
An analogous formulation of SPPA can be made for the Fr\'{e}chet Median problem.
Even more importantly, several versions of split proximal point algorithms have been shown to converge globally to the Fr\'{e}chet mean \cite{bacak2014computing}.  

One strategy for minimizing the Fr\'{e}chet function is use a split proximal point algorithm for global
search and switch to a local search procedure.
The motivation for switching to a local search procedure 
is that if local search is initialized close to the optimal solution then faster convergence can be achieved.
The local optimization problem is minimizing the Fr\'{e}chet function in a 
fixed orthant $\Or$. 
One feature of the local optimization problem is that the Fr\'{e}chet function is $C^\infty$ in the interior of each pre-multi-vistal cell, but not at points on shared faces of pre-multi-vistal facets.
The Fr\'{e}chet function is $C^1$ only when restricted to the interior of a maximal orthant. 
An analysis of differential properties of the Fr\'{e}chet function is presented in Sec. \ref{sec:DiffAnalysis}.

\section{Decomposition and Relative Optimality Theorems}\label{sec:main_results}

The main analytical results related to minimizing $F(X)$
for a given a set of trees $T^1,...,T^n$ in $\T_r$ are presented in this section, and proofs are presented in Sec. \ref{sec:DiffAnalysis}. 
Each tree $T^i$ induces pre-vistal facets on treespace and taken together the collection subdivides treespace into pre-multi-vistal facets where the Fr\'{e}chet function can be represented in a fixed form. On the shared faces of pre-vistal facets the Fr\'{e}chet function can be represented in multiple valid forms. 
At such points the value of the Fr\'{e}chet function and gradient are the same, but higher order derivatives can differ depending on which representation of the Fr\'{e}chet function is used. 
Thus careful consideration of the differential properties of the Fr\'{e}chet function at such points is necessary for the Directional Derivative Decomposition Theorem (Thm. \ref{thm:DirDerDecomposition})

\subsection{Decomposition Theorem}
Let $X$ and $Y$ be points in $\T_r$ such that $X$ and $Y$ share a pre-multi-vistal facet defined by geodesics from $X$ to $T^1,...,T^n$, $\V(T^1,\Or;\A^1,\B^1) \cap \ldots \cap \V(T^n,\Or;\A^n,\B^n)$.
If this is the case, then either (i) $X$ and $Y$ have the same topology, (ii) $X$ is a contraction of $Y$ or (iii) $Y$ is a contraction of $X$.
Assume that if the topologies of trees $X$ and $Y$ differ then $X$ is a contraction of $Y$, that is $\Or(X) \subseteq \Or(Y)$. 
Let $\Gamma(X,Y;\alpha)$,
where $0 \leq \alpha \leq 1$, be the point $\alpha$ proportion along the geodesic from $X$ to $Y$.
\begin{defn}\label{def:DirDer}
The \emph{directional derivative from $X$ to $Y$} is 
\begin{align}
F'(X,Y) =&
\lim_{\alpha \to 0}\frac{F(\Gamma(X,Y;\alpha))-F(X)}{\alpha}
\end{align}
\end{defn}

\begin{defn}
Let $\Or^\perp (X)$ be the orthogonal space to $\Or(X)$ at $X$, that is the union of all orthogonal spaces in all orthants containing $\Or(X)$.
\end{defn}
An example of an orthogonal space can be seen in the half-open book with three pages illustrated in Fig. \ref{open_book}. Consider a point $X$ on the spine of the half-open book, labeled $\{0,1\}|\{2,3,4\}$. The orthogonal space at $X$ in one page of the book is a copy of the positive real line, and the complete orthogonal space at $X$ is three copies of the positive real line, one copy identified with each page of the open book.

Thm. \ref{thm:DirDerDecomposition} states that the value of the directional derivative can be decomposed into
a contribution from the change in $F(X)$ resulting in adjusting positive
length edges in $X$, and a contribution from the 
change in $F(X)$ resulting in increasing the lengths of edges
from zero. 

\begin{thm}\label{thm:DirDerDecomposition}
(Decomposition Theorem for Directional Derivatives)
Let $X,Y\in \T_r$, with $\Or(X)\subseteq \Or(Y)$ and with $X$ and $Y$ in a common multi-vistal cell, $V_{XY}$, let $Y_X$ be the projection of $Y$ onto $\Or(X)$,
and let $Y_\perp$ be the projection of $Y$ onto $\Or^\perp(X)$ at $X$.
Then,
\begin{align}
F'(X,Y) = F'(X,Y_X)+F'(X,Y_\perp)
\end{align}
\end{thm}
\noindent See Sec. \ref{sec:DiffAnalysis} for proof and see supplemental material for an example.

\subsection{Orthant optimization}

Consider a variable tree $X \in \T_r$ and a fixed set of compatible edges $E$. 
The goal is to minimize the Fr\'{e}chet function but under the restriction that 
the topology of $X$ may only have edges from $E$.
Under this restriction, the geometric location of $X$ is restricted to the orthant defined by the set of edges $E$, $\Or(E)$. 
As the edge lengths of $X$ vary the geodesic from $X$ to $T^i$ will also vary, 
and the support sequence $(\A^i,\B^i)=(A^i_1,B^i_1),\ldots,(A^i_{k^i}, B^i_{k^i})$ will change whenever $X$ crosses
the boundary of a vistal cell. 
Local search can be formulated as the following convex optimization problem.

\setcounter{equation}{0}\label{opt:P1}
\noindent Objective
\begin{align}
\min \quad
& F(X) =\sum_{i=1}^n \left(\sum_{l=1}^{k^i} (\norm{A^i_l}+\norm{B^i_l})^2 + \sum_{e \in C^i} (|e|_X-|e|_{T^i})^2\right)
\end{align}
Constraints
\begin{align}
|e|_X & \geq 0 && \forall e \in E
\end{align}
The minimizer of this optimization problem, $X^*$, satisfies $F(X^*)\leq F(X)$ for all $X$ in $\Or(E)$.

\subsection{Optimality Qualifications}\label{sec:Optimality Qualifications}
There are two cases for the optimal solution $X^*$: either every edge in $X^*$ has a positive length or at least one edge in $X^*$ does not. 
If every edge of $X^*$ has positive length, then $X=X^*$ if and only if $\nabla F(X)=0$ because the Fr\'{e}chet function is continuously differentiable in the interior of $\Or(E)$. 
The optimality condition for a point
on a lower dimensional face of treespace can be expressed in terms of directional derivatives. In that case the optimality condition is
\begin{align}\label{eq:simpleOptCond}
F'(X,Y) & \geq 0  & \textrm{for all $Y$ such that $\Or(X) \subseteq \Or(Y)$}
\end{align}

By using Thm. \ref{thm:DirDerDecomposition} to separate the directional derivative into the contribution from the component of $Y$ in $\Or(X)$ and the component of $Y$ which is perpendicular to $\Or(X)$ the optimality condition can be expanded into a conditions on these independent terms, and then simplified by taking advantage of the ability to express $F’(X,Y_X)$ in terms of the gradient of $F$ at $X$ relative to $\Or(X)$.
The following optimality conditions are expressed in terms of the Fr\'{e}chet function (Def. \ref{def:FrGrad}) and the directional derivative of the Fr\'{e}chet Function (Def. \ref{def:DirDer}).
Let $[\nabla F(X)]_e$ denote the partial derivative of the Fr\'{e}chet function with respect to edge $e$, $\frac{ \partial F(X)}{\partial e}$, which is well defined when $|e|_X>0$. 
\begin{align} 
[\nabla F(X)]_e& = 0 &
\textrm{ for all $e: |e|_X > 0$} \label{decompOptCond1}\\
F'(X,Y) & \geq 0 & \textrm{for all $Y$ in $\Or(E)$ such that the component of $Y$ in $\Or(X)$ is 0} \label{decompOptCond2}
\end{align}
The local search problem i.e. identifying the minimizer
of the Fr\'{e}chet function on an orthant of treespace, $\Or(E)$, must have a unique solution because the Fr\'{e}chet function is strictly convex and $\Or(E)$ is a convex set.
Also, optimality conditions for the local search problem 
are only different from global optimality conditions in one aspect, 
which is that rather than requiring $F'(X,Y)\geq 0$ for all
$Y$ perpendicular to $\Or(X)$, it is only necessary
to consider the subset of such points, $Y$, which are in $\Or(E)$.

\subsection{Verifying optimality}
The focus of this section is developing an algorithm to answer the following question:
Given a point $X$ on a lower dimensional face of an orthant, $\Or(E)$, does there exist $Y$ such that $F'(X,Y)<0$?
To develop an algorithm to answer this question, the mathematical conditions of optimality from Sec. \ref{sec:Optimality Qualifications} are simplified into a collection of nested conditions. 

The decomposed local optimality conditions, Cond. (\ref{decompOptCond1}-\ref{decompOptCond2}) state: partial derivatives with respect to positive length edges must be zero, and directional derivatives for directions which introduce new edges must be non-negative, respectively.
Reformulating Cond. (\ref{decompOptCond2}) as an optimization problem will lead to a simplified set of optimality conditions. 
Bounding the minimum directional derivative below by zero, as in the following optimization problem, is equivalent to bounding all directional derivatives below by zero.
\begin{align}
f^* \geq 0,
\end{align}
where
\begin{align}
f^*=& \min_{\begin{array}{c} Y \neq X \\ (Y-X)\perp \Or(X) \end{array}} F'(X,Y)
\end{align}
The directional derivative of the Fr\'{e}chet function, $F'(X,Y)$ is a positively homogeneous function. Since detecting a negative directional derivative is the goal, it is sufficient to bound the directional derivative below on the intersection of a sphere around $X$ and the orthant $\Or(E)$. Here we choose to restrict the search domain to the unit simplex around $X$. Restriction to this polyhedral set is advantageous because $F'(X,Y)$ is a convex function of $Y$, Lem. \ref{lem:DirDerConvex}. Therefore, it is convex over a polyhedral subset of its domain.
Let $P$ be the vector of differences of edge lengths between $Y$ and $X$ with the component for edge $e$ having value $p_e = |e|_Y-|e|_X$. 
\begin{align}
f^*_1 = \text{ min }  & F'(X,Y) \label{minDirDer1}\\
\textrm{ s.t. } & \sum_{e \in E} p_e =1\label{minDirDer2}\\
& p_e = 0 \text{ for all $e$ s.t. $|e|_X > 0$}\label{minDirDer3}\\
& p_e \geq 0 \text{ for all $e$ s.t. $|e|_X = 0$}\label{minDirDer4}
\end{align}
The directional derivative is a once continuously differentiable function of $Y$ with respect to edges such that $p_e > 0$. However, this is generally not the case when $p_e=0$. Optimality qualifications for the optimization problem defining $f^*_1$ need to account for non-differentiability. 
Thm. \ref{thm:DirDerOfDirDer}, in the same spirit as the Decomposition Theorem for Directional Derivatives, Thm. \ref{thm:DirDerDecomposition}, provides a decomposition for directional derivatives of directional derivatives. 
The directional derivative, $F'(X,Y)$, is a $C^1$ function of $Y$ in $\Or(Y)$, as shown in Lem. \ref{lem:DirDerC1}. Therefore, when the positive edge lengths of $Y$ vary, the change in $F'(X,Y)$ can be quantified by partial derivatives of $F'(X,Y)$. 
The next theorem describes how the value of $F'(X,Y)$ changes in the case when lengths of edges in $Y$ are increased from zero.

\begin{thm}\label{thm:DirDerOfDirDer}
Let $X$, $Y$, and $Y'$ be three points in treespace such that $X$ is a contraction of $Y$ and $Y$ is a contraction of $Y'$, that is $\Or(X) \subset \Or(Y) \subset \Or(Y')$. 
Let $Y^\alpha=\Gamma(Y,Y';\alpha)$.
Let $Y'_{\perp \Or(Y)}$ be the component of $Y'$ which is perpendicular to the orthant containing $Y$, $\Or(Y)$. 
When situated as described above, the directional derivative of the directional derivative, $F'(X,Y)$, when $Y$ varies in the direction of $Y'$, that is 
\begin{align}
\lim_{\alpha \to 0} \frac{F'(X,Y^\alpha)-F'(X,Y)}{\alpha},
\end{align}
is equal to the sum of a term contributed by changes which are parallel to the axes of $\Or(Y)$ and the directional derivative of the Fr\'{e}chet function at $Y$ in the direction of $Y'_{\perp \Or(Y)}$, $F'(Y,Y'_{\perp \Or(Y)})$.
The following expression gives a decomposition of the rate of change in the directional derivative into a term contributed from changing positive length edges, $|e|_Y>0$ and a term contributed by increasing zero length edges, $|e|_Y=0$. This decomposition is
\begin{align}
\lim_{\alpha \to 0} \frac{F'(X,Y^\alpha)-F'(X,Y)}{\alpha}=\sum_{e\in E_Y}{(|e|_{Y'}-|e|_Y)}[\nabla F'(X,Y)]_e+F'(Y,Y'_{\perp \Or(Y)}).
\end{align}
\end{thm}
\noindent See Sec. \ref{sec:DiffAnalysis} for proof.

With the decomposition from Thm. \ref{thm:DirDerOfDirDer} one can express the optimality conditions for 
the optimization problem defined by (\ref{minDirDer1}-\ref{minDirDer4}) as follows.
The optimal point $Y^1$ makes the vector of partial derivatives of the directional derivative, $\nabla_Y F'(X,Y)$, point away from the origin of the orthant and have a zero projection onto the set defined by $\sum_{e\in E} p_e = 1$.
And, in any direction orthogonal to the orthant containing $Y^1$ the directional derivative from $Y^1$ is positive, that is $F'(Y^1,Y'_{\perp \Or(Y^1)})\geq 0$. 
Let $E^1 \subset E$ be the set of edges such that $|e|_{Y^1}-|e|_X>0$.
The optimality condition for a point $X^*$ which minimizes the Fr\'{e}chet function $F(X)$ in an orthant $\Or(E)$ can be expanded to 
\begin{align}
\nabla F(X)=0\label{ExpCond1L1}\\
F'(X,Y^1) \geq 0 \label{EpCond1L2}\\
\nabla_Y F'(X,Y^1) \perp \{Y|\sum_{e \in E^1} p_e =1\}\label{ExpCond1L3}\\
F'(Y^1,Y)\geq 0 \text{ for all $Y \neq Y^1$ s.t. $(Y-Y^1)\perp \Or(Y^1)$}\label{ExpCond1L4}.
\end{align}
This expanded optimality condition is obtained by reformulating a lower bound on the directional derivative
of the Fr\'{e}chet function as an optimization problem and applying 
the decomposition from Thm. \ref{thm:DirDerOfDirDer}. 
The same strategy can be applied first on the condition on line (\ref{ExpCond1L4}) and recursively on the resulting expansion.
This strategy yields the following theorem. 
\begin{thm}\label{thm:RecursiveOptimalityConditions}
Consider trees $Y^0,\ldots,Y^k$ such that (i) $\Or(Y^0)\subset \ldots \subset \Or(Y^k)=\Or(E)$, and (ii) $Y^{i+1}-Y^i \perp \Or(Y^i)$ for $i=0,\ldots,k-1$.
Let $E^i$ be the set of positive edges in $Y^i$ for $i=0,\ldots,k$. 
Define a set of edge length difference vectors $P^i$ for $i=1,...,k$ with the component for edge $e$ having value $p^i_e=|e|_{Y^i}-|e|_{Y^{i-1}}$.
Denote the unit simplex in the orthant $\Or(E^i\setminus E^{i-1})$ by $\Delta^i=\{P \in \Or(E^i\setminus E^{i-1})|\sum p^i_e=1\}$.
The minimizer of $F(X)$ in $\Or(E)$ is $Y^0$ if and only if
\begin{align}
\nabla F(Y^0)=0 &\\
F'(Y^{i-1},Y^i)\geq 0 & \textrm{
for $i=1,\dots,k$}\\
\nabla F'(Y^{i-1},Y^i) \perp \Delta^i &\textrm{
for $i=1,\dots,k$}.
\end{align}
\end{thm}
\noindent See Sec. \ref{sec:DiffAnalysis} for proof.

\section{Methods for optimizing edge lengths}\label{sec:IntPointMethods}

This section contains the fundamentals for an iterative local search algorithm: initialization, an improvement method, and a method to verify optimality.

\subsection{A damped Newton's method}\label{sec:DampedNewton}
Alg. \ref{alg:interiorPoint} is designed to find approximately optimal edge lengths for a fixed tree topology.
Detailed explanations for the steps 
of this algorithm are in the following subsections.

\noindent{\bf{$\delta$-$\epsilon$ optimality conditions}}\\
Conditions for a point $X$ on the interior of an orthant to be approximately optimal are:
\begin{equation}\label{eq:ApproxOptimal}
\begin{array}{c c c c}
& & \abs{\left[\nabla F(X)\right]_e} < \delta & \textrm{ for all $e: |e|_X > \epsilon$} \\
\left[\nabla F(X)\right]_e \geq 0 & \textrm{ or } & \abs{\left[\nabla F(X)\right]_e} < \delta& \textrm{ for all $ e: |e|_X  < \epsilon$}
\end{array}
\end{equation}
If the $\delta$-$\epsilon$ optimality conditions are satisfied then $F(X^*)$ will not differ much
from the Fr\'{e}chet function value when the lengths of edges with positive derivatives
are set to 0.

\begin{algorithm}
\caption{Interior point algorithm for optimal edge lengths}
\label{alg:interiorPoint}
\begin{algorithmic}
\STATE{{\bfseries input:} $T^1,T^2,\ldots,T^n,X^0 \in \T_r$, $\epsilon>0$, $\delta>0$, $0<c<1$}
\WHILE{$\delta$-$\epsilon$ optimality conditions (\ref{eq:ApproxOptimal}) are not satisfied  }
\STATE{ compute a descent direction $P$ (Sec. \ref{sec:NewtonSteps})}
\STATE{ find a feasible step-length, $\alpha$, satisfying decrease condition (Sec. \ref{sec:StepLength}) }
\STATE{ {\bfseries let} $X^{k+1}=X^k+\alpha P$ }
\STATE { {\bfseries if} $|e| < \epsilon$ {\bfseries then} remove edge $e$ from tree $X$  }
\ENDWHILE
\end{algorithmic}
\end{algorithm}

\subsubsection{Newton steps}\label{sec:NewtonSteps}
A successful iterative algorithm will make substantial progress to an optimal point.
This can be achieved using a modified Newton's method. 
Newton's method uses a descent vector which points to the minimizer of a quadratic approximation of the objective function. The quadratic approximation in Newton's method uses the first three terms of the Taylor expansion of $F(X)$.
For the Fr\'{e}chet function the entries of the Hessian matrix of second order partial derivatives is given in Def. \ref{def:FrHess}.

The Hessian matrix is positive definite because the Fr\'{e}chet function
is strictly convex. 
Let $\left[\nabla^2 F(X)\right]_{ee'} = \frac{\partial ^2 F(X)}{\partial e \partial e'}$, which is well defined when $|e|_X>0$ and $|e'|_X>0$.
When the Hessian and gradient are well defined,
the second order Taylor approximation is 
\begin{align}
g(X; p)=& F(X)+ \sum_{e \in E}p_e[\nabla F(X)]_e
&+ \sum_{e \in E} \sum_{e' \in E} p_e p_{e'} \left[\nabla^2 F(X)\right]_{ee'}.
\end{align}
The minimizer in $p$ of $g(X;p)$ is the Newton vector $p^N = - \nabla F(X) (\nabla^2 F(X))^{-1} $. 

\subsubsection{Choosing a step length}\label{sec:StepLength}
Taking a full step along the Newton direction
minimizes the quadratic approximation of the 
Fr\'{e}chet function. However, taking a full step may result in a new point which actually has a larger
Fr\'{e}chet function value or that may
be beyond orthant boundaries.

The first precaution is to calculate the maximum step
length $\alpha_0$ such that $|e|_{k+1}=|e|_k+\alpha_0 p_e \geq 0$ for all $e$.  
If $\alpha_0 \leq 1$ then let $\alpha = \alpha_0 c_0$
where $0<c_0 <1$.

Choosing
step-length which satisfies the following \emph{sufficient decrease condition}
will ensure  a substantial decrease
in the objective function value
at every step.
Let $0<c_1<1$.
\begin{align}\label{eq:SufficientDecrease}
F(X^k+\alpha p) \leq F(X^k)+c_1 \alpha \sum_{e \in E_{X^k}} { [\nabla F(X)]_e p_e}
\end{align}
The \emph{curvature condition}, which rules out
unacceptably short steps, requires the step-length, $\alpha$, to satisfy
\begin{align}
\sum_{e \in E_{x^k}} [\nabla F(X^k+\alpha p)]_e p_e \geq c_2 \sum_{e \in E_{X^k}} { [\nabla F(X)]_e p_e}
\end{align}
for some constant $c_2$ in the interval $(c_1,1)$.

\subsubsection{Initialization}
For initializing an interior point search, any point in $\Or(E)$ would suffice, but it is preferential to start with a good guess for edge lengths. 
The global search algorithms presented in \cite{Miller2015,bacak2014computing,bacak2014convex} could provide a starting point for a local search.
One good start
strategy
can be derived by noticing that the Fr\'{e}chet function can be separated into a quadratic part and a part involving sums of norms.
\begin{align}
F(X)
&  =\sum_{i=1}^n \sum_{l=1}^{k^i} \norm{A^i_l}^2+2\norm{A^i_l}\norm{B^i_l} +\norm{B^i_l}^2 + \sum_{e \in C^i} (|e|_X^2-|e|_{T^i})^2
\end{align}
The only terms that cannot be expressed in a quadratic function of the edge lengths are collected into function
\begin{align}
S(X) 
&= \sum_{i=1}^n \sum_{l=1}^{k^i} 2 \norm{A^i_l}\norm{B^i_l}
\end{align}

Subtracting $S(X)$ from $F(X)$ yields a quadratic function,
\begin{align}
Q(X) = F(X)-S(X) & = \sum_{i=1}^n \sum_{l=1}^{k^i} \norm{A^i_l}^2+\norm{B^i_l}^2 +\sum_{e \in C^i} (|e|_X-|e|_{T^i})^2,
\end{align}

The minimizer of $Q(X)$, $X^*_Q$, can be easily found by solving $\nabla Q(X)=0$; the solution is
\begin{align}
|e|_{X^*_Q}=\frac{\sum_{i=1}^n |e|_{T^i}}{n}.
\end{align}
The optimal value $|e|_{X^*_Q}$ is non-negative, and if $e$ has a positive length in any of $T^1,\ldots,T^n$, then $|e|_{X^*_Q}$ is positive.

The gradient of $S(X)$ is non-negative at any feasible $X$.
Therefore, at any $X$ such that $\nabla Q(X) = 0$, the gradient of the Fr\'{e}chet function is greater than or equal to zero at every coordinate, $\nabla F(X)=\nabla Q(X)+\nabla S(X)>=0$.
Which implies that
the optimal edge lengths in $X^*$ must be no larger than the edge lengths in $X^*_Q$ i.e. the optimal solution is in the closed box
\begin{align}\label{eqBox}
0\leq |e|_X & \leq |e|_{X^*_Q} & & \forall e \in E.
\end{align}
Thus a reasonable starting point for search inside orthant $\Or(E)$ would be $X^*_Q$, or any point strictly inside the box defined by Eq. (\ref{eqBox}).

\subsection{Iterative algorithm for verifying optimality in a closed orthant}\label{sec:OrthantOpt}

The optimality condition in Thm. \ref{thm:RecursiveOptimalityConditions} is the logical basis for the Alg.\ref{alg:OrthantOpt}, which finds the minimizer of the Fr\'{e}chet function $F(X)$ in the closure of a fixed orthant, $\Or(E)$. 

\begin{algorithm}
	\caption{Algorithm to minimize Fr\'{e}chet sum of squares in orthant $\Or(E)$}
	\label{alg:OrthantOpt}
	\begin{algorithmic}
	
	\STATE{ {\bfseries input:} $E$; $T^1,T^2,\ldots,T^n$; $\epsilon,\delta>0$}
	\STATE{ {\bfseries initialize:} $i,k=1$; $E^0=$ optimal star tree; $E^1=E\setminus\{\textrm{leaf edges}\}$; $\epsilon^0,\epsilon^1=\epsilon$ }
	\WHILE{$i\geq 1$ }
	\STATE{ Find $Y^* \in \Or(E^i)$ which approx. minimizes $F'(Y^{i-1},Y)$, }
	\STATE{ where $Y-Y^{i-1}$ is orthogonal to $\Or(E^0\cup\ldots \cup E^{i-1})$}
	\WHILE{ approx. optimality conditions with $\epsilon^i$ and $\delta$ are not satisfied }
	\STATE{ damped Newton}
	\ENDWHILE
	\STATE{ {\bfseries if} there are some zero length edges, $S$}
	\STATE{ \hspace*{1cm}  $(E,Y,\epsilon)^{i+1}=(E,Y,\epsilon)^i,\ldots,(E,Y,\epsilon)^{k+1}=(E,Y,\epsilon)^k$; $k=k+1$ }
	\STATE{ \hspace*{1cm} $E^i=E^i\setminus S$; $|e|_{Y^i}=0$ if $e \in S$ }
	\STATE{ \hspace*{1cm} $i=i+1$}
	\STATE{ {\bfseries else}}
	\STATE{ \hspace*{1cm}  {\bfseries if} $F'(Y^{i-1},Y^*)< 0$}
	\STATE{ \hspace*{2cm} find step size $\alpha$, along $\Gamma(Y^{i-1},Y^*)$; }
	\STATE{  \hspace*{2cm}  $E^{i-1}=E^{i-1}\cup E^i$; $Y^{i-1}=Y^{i-1}+\alpha Y^*$; $\epsilon^{i-1}=\alpha c$ where $0<c<1$; }
	\STATE{ \hspace*{2cm}  $(E,Y,\epsilon)^{i}=\emptyset$; $(E,Y,\epsilon)^{i}=(E,Y,\epsilon)^{i+1},\ldots, (E,Y,\epsilon)^{k-1}=(E,Y,\epsilon)^k$}
	\STATE{  \hspace*{2cm} $i=i-1$, $k=k-1$}
	\STATE{ \hspace*{1cm}  {\bfseries else}}
	\STATE{ \hspace*{2cm} $i=i-1$}
	\STATE{ \hspace*{1cm}  {\bfseries end if}}
	\STATE{ {\bfseries end if}}
	\ENDWHILE
	\STATE{ Find the minimizer of $F(X)$ in $\Or(E^0)$.}
\end{algorithmic}
\end{algorithm}

\section{Differential analysis of the Fr\'{e}chet function in treespace}\label{sec:DiffAnalysis}

Analysis of how $F(X)$ changes with respect to $X$ provides 
useful insights for designing fast optimization algorithms.
This analysis is aimed at providing summaries for how the value of $F(X)$
changes with respect to $X$. 


The results in this section are summarized as follows:
Cor. \ref{cor:DirDerTangentSpace} gives
the value of the directional derivative when $\Or(X)=\Or(Y)$ and
Thm. \ref{thm:DirDerValue} gives the value of the directional
derivative when $\Or(X) \subseteq \Or(Y)$,
when assuming $Y$ is contained in the interior of a
multi-vistal facet.
In Lem. \ref{lem:DDWellDefined} we show that when $Y$ is on a multi-vistal face 
the value of the directional derivative 
can be expressed equivalently using any of the
representations for the geodesics from $T^1,\ldots,T^n$ to $Y$.
In Lem. \ref{lem:DirDerContinuous} and Lem. \ref{lem:DirDerConvex}
we show that the directional derivative is continuous and convex 
with respect to $Y$ when $\Or(X) \subseteq \Or(Y)$.  The proofs of Thm. \ref{thm:DirDerDecomposition}, \ref{thm:DirDerOfDirDer}, and \ref{thm:RecursiveOptimalityConditions} are presented at the end of Sec. \ref{sec:DiffAnalysis}.

When both $X$ and $Y$ are in the relative interior of the same maximal orthant of treespace,
where the gradient at $X$ is well defined in $\Or(Y)$, the directional derivative can be expressed
in terms of the gradient at $X$ inside $\Or(Y)$. However when $\Or(X) \subset \Or(Y)$, the
gradient at $X$ is not well defined in $\Or(Y)$. Analysis of the directional derivative in the
later situation, which is one of the main focuses of this section, is important because it
facilitates concise specification of optimality conditions and an efficient algorithm for verifying that a point on a lower dimensional face of an orthant $\Or$ is the minimizer of the Fr\'{e}chet function within $\Or$. 
\begin{thm}\label{thm:Grad}
\cite[Cor. 4.1]{Miller2015} The gradient of $F$ is well defined on the interior of every maximal orthant $\Or$.
\end{thm}
\noindent Idea of proof.
The Fr\'{e}chet function is smooth in each multi-vistal facet, and it can be shown that the gradient
function has the same value in every multi-vistal facet containing $X$ in the interior of $\Or$. Therefore the gradient is well-defined on the interior of $\Or$. 

\begin{cor}\label{cor:DirDerTangentSpace}
When $X$ and $Y$ are in the same maximal orthant the value of directional derivative from $X$ to $Y$ can be expressed
in terms of the gradient at $X$, and the differences in edge lengths $p_e =|e|_Y-|e|_X$, as
\begin{align}
F'(X,Y) = \sum_{e\in E_X} p_e \left[\nabla F(X)\right]_e
\end{align}
\end{cor}
\begin{proof}
Expression of a directional derivative of a smooth function in terms of its gradient is a standard technique in calculus.
\end{proof}

The gradient may not be well-defined on a lower-dimensional orthant of treespace.
For a point on a lower dimensional orthant of
treespace, a well-defined analogue to the gradient is the \emph{restricted gradient}. 

\begin{defn}\label{def:FrGrad}
Let $(A^i_1,B^i_1),\ldots,(A^i_{k^i},B^i_{k^i})$ be a support
sequence for the geodesic from $X$ to $T^i$. The \emph{restricted gradient}
is the vector of partial derivatives which correspond to points $Y$ with $\Or(X) \subseteq \Or(Y)$ and $Y-X$ parallel to the axes of $\Or(X)$.
If $|e|_X >0$ then
\begin{align}
\left[\nabla F(X)\right]_e &=
\frac{\partial F(X) }{\partial e}\\
& = \lim_{\Delta e \to 0} \frac{F(X+\Delta e)-F(X)}{\Delta e}\\
&= \sum_{i=1}^n\left\{ \begin{array}{ll}
|e|_X \left(1+ \frac{||B^i_l||}{||A^i_l||}\right) & \textrm{if $e \in A_l^i$} \\
\left(|e|_X-|e|_{T^i} \right) & \textrm{if $e \in C^i$}\\
\end{array} \right.
\end{align}
and  if $|e|_X = 0$ then $\left[\nabla F(X)\right]_e=0$. 
\end{defn}
When $X$ is on the interior of a maximal orthant of treespace then the restricted gradient is the same as the gradient.
Note that in the case when $A^i_l=\{e\}$, $|e|_X \left(1+ \frac{||B^i_l||}{||A^i_l||}\right)=|e|_X + \norm{B^i_l}$.

Second order derivatives will be used in calculating Newton directions
in Sec. \ref{sec:IntPointMethods}.

\begin{defn}\label{def:FrHess}
Let $X$ be a point in the interior of a multi-vistal cell relative to an orthant of treespace, $\Or$. 
The restricted Hessian on $\Or$ is the matrix of second order partial derivatives
with entries having the following values:
\begin{equation}\label{FrHess}
\left[\nabla^2 F(X)\right]_{ef} = 2 \sum_{i=1}^{r} { \left\{ \begin{array}{ll}
1+\frac{\norm{B_l^i}}{\norm{A_l^i}} - \frac{\norm{B^i_l}}{\norm{A^i_l}^3}x_e^2 & \textrm{if $e=f$, $e \in A^i_l$, $|A^i_l|>1$} \\
1& \textrm{if $e=f$, $e \in A^i_l$, $A^i_l=\{e\}$} \\
1 & \textrm{if $e=f$ $e \in C^i$} \\
-\frac{\norm{B^i_l}}{\norm{A_l^i}^3}x_e x_f & \textrm{if $e \neq f$ $e,f \in A_l^i$} \\
0 & \textrm{otherwise}
\end{array} \right. }
\end{equation}
If either $|e|_X=0$ or $|f|_X=0$ then $\left[\nabla^2 F(X)\right]_{ef}=0$. 
\end{defn}

\begin{thm}
The value of the restricted gradient at a point $X$ can be expressed equivalently using the algebraic form of the Fr\'{e}chet function from any of the multi-vistal facets containing $X$.
\end{thm}
\begin{proof}
The restricted gradient has the same value
using any of the valid support sequences defined by vistal cells on the relative interior of $\Or$.

We now verify that at $X$ the gradient of $d^2(X,T^i)$ is the same for every valid support and signature. The gradient of $d^2(X,T^i)$ for the support $(\A,\B)$ is given as follows. Let the variable length of edge $e$ in $X$ be written as $x_e$.
\begin{equation}\label{graddistance}
\frac{\partial d^2(X,T^i)}{\partial x_e}= \left\{ \begin{array}{ll}
2 \left(1+ \frac{||B_l^i||}{||A_l^i||}\right)x_e & \textrm{if $e \in A^i_l$} \\
2 \left(x_e-|e|_T^i \right) & \textrm{if $e \in C^i$}
\end{array} \right.
\end{equation}
In this paragraph we focus on the behavior of the geodesic, $\Gamma^i$, from $X$ to tree, $T^i$.
We drop the superscript $i$ to make the notation less cumbersome when comparing two valid supports for the same geodesic. 
The geodesic $\Gamma$ has a unique support$(\A,\B)$ satisfying (P3) with strict inequalities, 
\begin{equation}\label{facetcondition}
\frac{\norm{A_1}}{\norm{B_1}} < \frac{\norm{A_2}}{\norm{B_2}} < \ldots < \frac{\norm{A_k}}{\norm{B_k}}.
\end{equation}
From \cite[Sec. 3.2.2]{Miller2015}, any other support $(\A',\B')$ for $\Gamma$ must have a signature $\mathcal{S}'$ in $(P3)$ with some equality subsequences. Suppose that $A_j'$ and $B_j'$ are in some equality subsequence satisfying $(P2)$ with $B_j'$ containing the edge $e$. Then for the support pair $A_i$ and $B_i$ such that $B_i$ contains $e$, it must hold that $\frac{\norm{A_j'}}{\norm{B_j'}} = \frac{\norm{A_i}}{\norm{B_i}}$. Now we can see that $ \left(1+ \frac{||B_j'||}{||A_j'||}\right)x_e = \left(1+ \frac{||B_i||}{||A_i||}\right)x_e$, and that the gradient of $d^2(X,T^i)$ is the same on every multi-vistal facet containing $X$ on the relative interior of $\Or$.
\end{proof}

Now we extend the results for directional derivatives to the situation when $\Or(X) \subset \Or(Y)$.
\begin{thm}\label{thm:DirDerValue}
Suppose that $Y$ lies in the interior of multi-vistal facet $V_Y$,
and $X$ is some point in $V_Y$. Let $(A^i_1,B^i_1),\ldots, (A^i_m,B^i_m)$ be the support 
pairs for the geodesic from $Y$ to $T^i$ and let $C^i$ be the set of edges
in $Y$ which are common in $T^i$. Let $E_X$ be the set of edges with positive lengths in $X$. 
Let $P$ be the vector with components $p_e = |e|_Y-|e|_X$ so that $\Gamma(X,Y;\alpha)=X+\alpha P$, and let $Z_\alpha:=\Gamma(X,Y;\alpha)$. 
Then the value of directional derivative from $X$ to $Y$ is 
\begin{displaymath}
F'(X,Y)=\sum_{e \in E_X} p_e \left[\nabla F(X)\right]_e + 2 \sum_{i=1}^n \left ( \sum_{l: \norm{A^i_l}_X = 0} \left(  \norm{A^i_l}_P\norm{B^i_l}_{T^i} \right) -\sum_{e \in C^i \setminus E_X} p_e|e|_{T^i}  \right).
\end{displaymath} 
Where $\norm{A}_X$, means use the edge length mapping from tree $X$ in evaluating the norm of the set of edges $A$, that is $\norm{A}_X = \sqrt{\sum_{e \in A} |e|_X^2}$.
\end{thm}
\begin{proof}
Let $Z_\alpha$ be a point on the geodesic segment between $X$ and $Y$. The length of edge $e$ in $Z_\alpha$ be $|e|_Z = |e|_X + \alpha p_e$.
The Fr\'{e}chet function is the sum of squared distances from a variable point to each of the data points $T^1,\ldots, T^n$, so the directional derivative of the Fr\'{e}chet function
is the sum over the indexes of the data points of the directional derivatives of
the square distances.
\begin{align}
F'(X,Y) &=
\lim_{\alpha \to 0}\frac{F(Z_\alpha)-F(X)}{\alpha}\\
& = \lim_{\alpha \to 0} \frac{\sum_{i=1}^n d^2(Z_\alpha,T^i)-\sum_{i=1}^n d^2(X,T^i)}{\alpha}\\
& = \sum_{i=1}^n \left ( \lim_{\alpha \to 0} \frac{d^2(Z_\alpha,T^i)-d^2(X,T^i)}{\alpha} \right )
\end{align}
For a set of edges $A$, let $\norm{A}_X = \sqrt{ \sum_{e \in A} |e|_X}$.  If an edge $e$ has zero length in a tree, $X$, or is compatible with $X$ but not present then take $|e|_X$ to be $0$.
The squared distance from $Z_\alpha$ to $T^i$ can be expressed as 
\begin{align}
d^2(Z_\alpha,T^i) & = \sum_{l=1}^{k^i} (\norm{A^i_l}_{Z_\alpha}+\norm{B^i_l})^2+\sum_{e\in C^i} (|e|_{Z_\alpha}-|e|_{T^i})^2
\end{align}
The squared distance has three types of terms: a term representing
the contribution from common edges, terms for support pairs with $\norm{A^i_l}_X>0$, and terms for support pairs with $\norm{A^i_l}_X=0$.
In the first two cases the gradient is well-defined, and taking the inner-product of the directional vector and the gradient will yield their contributions
to the directional derivative. In the third case the gradient is undefined, and its value will be obtained by analyzing the limit directly as follows.
\begin{align}
&\left ( \lim_{\alpha \to 0} \frac{\sum_{l=1}^{k^i} (\norm{A^i_l}_{Z_\alpha}+\norm{B^i_l})^2-\sum_{l=1}^{k^i} (\norm{A^i_l}_X+\norm{B^i_l})^2}{\alpha} \right )\label{eq:DirDerSqDist2}
\end{align}
Bringing out the sum and canceling in the numerators yields,
\begin{align}
\sum_{l=1}^{k^i} \lim_{\alpha \to 0} \frac{\norm{A^i_l}^2_Z-\norm{A^i_l}^2_X+2\norm{B^i_l}\left(\norm{A^i_l}_Z-\norm{A^i_l}_X\right)}{\alpha}
\end{align}
The limit of the fraction can be split into the sum of two limits, 
\begin{align}
&\lim_{\alpha \to 0} \frac{\norm{A^i_l}^2_Z-\norm{A^i_l}^2_X+2\norm{B^i_l}\left(\norm{A^i_l}_Z-\norm{A^i_l}_X\right)}{\alpha}\\
=&\lim_{\alpha \to 0} \frac{\norm{A^i_l}^2_Z-\norm{A^i_l}^2_X}{\alpha}+\lim_{\alpha \to 0}\frac{2\norm{B^i_l}\left(\norm{A^i_l}_Z-\norm{A^i_l}_X\right)}{\alpha} \label{eq:LimitFrac1}
\end{align}
If every edge in $A^i_l$ has length zero in $X$, and thus $\norm{A^i_l}_X=0$,
the limit on the left is 0 and the limit on the right simplifies to 
\begin{align}
2\norm{B^i_l} \norm{A^i_l}_P \label{eq:Incomp2}
\end{align}

The partial derivative of the squared distance from $X$ to $T^i$ with respect to the length of edge $e$, that is the component for edge $e$ in the restricted gradient vector, is
\begin{align}
[\nabla d^2(X,T^i)]_e = \left\{ \begin{array}{ll}
|e|_X \left(1+ \frac{||B^i_l||}{||A^i_l||}\right) & \textrm{if $e \in A_l^i$} \\
\left(|e|_X-|e|_{T^i} \right) & \textrm{if $e \in C^i$}
\end{array} \right.
\end{align}
The directional derivative of the squared distance simplifies to 
\begin{align}
&\lim_{\alpha \to 0} \frac{d^2(Z,T^i)-d^2(X,T^i)}{\alpha}\\
&=\sum_{e \in  E_X} p_e [\nabla d^2(X,T^i)]_e+2\sum_{l:\norm{A^i_l}_X = 0}\left( \norm{B^i_l} \sqrt{\sum_{e\in A^i_l} p_e^2} \right)-2\sum_{e \in C^i \setminus E_X} p_e |e|_{T^i}
\end{align}
Summing the directional derivatives of the squared distances over $T^1,\ldots, T^n$ yields
the expression for the value of the directional derivative in the theorem.
\end{proof}

Each tree $T^i$ induces pre-vistal facets on treespace and taken together the collection subdivides treespace into pre-multi-vistal facets where the Fr\'{e}chet function can be represented in a fixed form. On the shared faces of pre-vistal facets the Fr\'{e}chet function can be represented in multiple valid forms. 
At such points the value of the Fr\'{e}chet function and gradient are the same, but higher order derivatives can differ depending on which representation of the Fr\'{e}chet function is used. 
In the following lemma the directional derivative is determined to be well-defined at such points.
We now extend the results to the situation where $Y$ is allowed to be
on a vistal face.
In this situation there can be multiple
valid support sequences for the geodesics from $Y$ to $T^1, \ldots, T^n$.
\begin{lem}\label{lem:DDWellDefined}
Suppose that $X$ and $Y$ are in the same multi-vistal facet, $\V$, and
that $Y$ is on a face of $\V$ on the interior of an orthant.  The value of the directional derivative can be expressed equivalently using any valid support sequences for the geodesics from $Y$ to $T^1,\ldots,T^n$.
\end{lem}
\begin{proof}
The form of $F'(X,Y)$ is constant within an open multi-vistal facet, and changes
at boundaries of vistal facets.
When $Y$ reaches the boundary of a vistal facet, that is either at least one of the (P2)
constrains reaches equality, at least one of the (P3) constraints reaches equality, 
or when the length of an edge reaches zero or increases from zero, this is called
the collision of $Y$ with the boundary of that vistal facet. 
A point $T^i$,  and associated geodesic $\Gamma(T^i,Y)$ are said to generate the vistal facet collision.
When $Y$ collides with a (P2) boundary of a vistal facet at least two support pairs for the geodesic merge; and when $Y$ collides with a (P3) boundary 
at least two support pairs for the geodesic
could be split in such a way that the resulting support is valid. 
In either case there are at least two valid forms for the geodesic.
Let $(C_1,D_1),(C_2,D_2)$ be support pairs which are formed from a partition of the support pair $(A^i_l,B^i_l)$, such that either of the following support sequences for the geodesic from $Y$ to $T^i$ is valid: $(A^i_1,B^i_1),\ldots,(A^i_l,B^i_l),\ldots,(A^i_m,B^i_m)$ or $(A^i_1,B^i_1),\ldots,(C_1,D_1),(C_2,D_2),\ldots,(A^i_m,B^i_m)$; and $ \frac{\norm{C_1}}{\norm{D_1}}=\frac{\norm{C_2}}{\norm{D_2}}=\frac{\norm{A^i_l}}{\norm{B^i_l}} $. 
Rescaling the lengths of edges in $A^i_l$ does not change the form of the geodesic for small $\alpha$ and $l \leq k$. Parameterizing the lengths of edges in terms of $\alpha$ and canceling $\alpha$ yields  $ \frac{\sqrt{\sum_{e\in C_1} p_e^2}}{\norm{D_1}}=\frac{\sqrt{\sum_{e\in C_2} p_e^2}}{\norm{D_2}}=\frac{\sqrt{\sum_{e\in A^i_l} p_e^2}}{\norm{B^i_l}} $. That fact, and the fact that $C_1 \cup C_2$ partition $A^i_l$, and $D_1 \cup D_2$ partition $B^i_l$ implies that 
$\norm{D_1}\sqrt{\sum_{e\in C_1} p_2^2} + \norm{D_2} \sqrt{\sum_{e\in C_2} p_e^2}=\norm{B^i_l}\sqrt{\sum_{e\in A^i_l} p_e^2}$. Thus the directional derivative is continuous across vistal facet boundaries from (P2) and (P3) constraints. 
\end{proof}

Now we extend the results for directional derivatives to directions issuing from $X$ to points $Y$ in a small enough radius such that $\Or(X) \subseteq \Or(Y)$ 
and $X$ and $Y$ share a multi-vistal facet.

\begin{lem}\label{lem:DirDerContinuous}
The directional derivative, $F'(X,Y)$, is a continuous function of $Y$ over the set of $Y$ such that $\Or(X) \subseteq \Or(Y)$ and $X$ and $Y$ share a vistal facet.
\end{lem}
\begin{proof}
The directional derivative is a continuous function at the faces of orthants because when an edge length $|e|_Y$ increases from zero its contribution to $F'(X,Y)$ is a continuous function which starts at the value zero. Thus, when the topology of $Y$ changes $F'(X,Y)$ changes continuously as a function of the edge lengths. 
\end{proof}

The following lemma is used in the proof of Lem. \ref{lem:DirDerConvex}, and was also discovered independently by Megan Owen \cite{Owen2014private}.

\begin{lem}\label{lem:LocalGeoTriangle}
Let $Y^0$ and $Y^1$ be points in $\T_r$ such that $\Or(X) \subseteq \Or(Y^0)$
and $\Or(X) \subseteq \Or(Y^1)$.
Let $Y^t=\Gamma(Y^0,Y^1;t)$ be the point which is proportion $t$ along the geodesic from $Y^0$ to $Y^1$.
The point which is $\alpha$ proportion along the geodesic from $X$ to $Y^t$ is $t$ proportion along the geodesic between the point $\Gamma_{X,Y^0}(\alpha)$ and the point $\Gamma_{X,Y^1}(\alpha)$, that is $\Gamma(X,Y^t;\alpha)=\Gamma(\Gamma(X,Y^0;\alpha),\Gamma(X,Y^1;\alpha);t)$.
\end{lem}
\begin{proof}
Let $Y^0(\alpha) = \Gamma_{X Y^0} (\alpha)$ and let $Y^1(\alpha) = \Gamma_{X Y^1} (\alpha)$.
Let $C = E_{Y^0(\alpha)} \cap E_{Y^1(\alpha)}$. By definition $E_X \subseteq C$.
The length of $e$ in $Y^0(\alpha)$ is
\begin{align}\label{eq:LengthY0}
|e|_{Y^0(\alpha)}=
\left\{ \begin{array}{ll}
|e|_X + \alpha |e|_{Y^0} & \textrm{if $e \in C$} \\
\alpha |e|_{Y^0} & \textrm{if $e \in E_{Y^0} \setminus C$}\\
\end{array} \right.
\end{align}
and the length of $e$ in $Y^1(\alpha)$ is
\begin{align}\label{eq:LengthY1}
|e|_{Y^1(\alpha)}=
\left\{ \begin{array}{ll}
|e|_X + \alpha |e|_{Y^1} & \textrm{if $e \in C$} \\
\alpha |e|_{Y^1} & \textrm{if $e \in E_{Y^1} \setminus C$}\\
\end{array} \right.
\end{align}
A geodesic support sequence which is valid for the geodesic between $Y^0$ and $Y^1$ is valid for the geodesic between $Y^0(\alpha)$ and $Y^1(\alpha)$.
The incompatibilities of edges in $A$ and $B$ are the same for any $\alpha$.
Suppose that a support sequence satisfies (P2) and (P3) for some $\alpha$. Factoring out $\alpha$ from the numerators and denominators of the $(P2)$ and $(P3)$ ratios reveals that the combinatorics of the geodesic between
$Y^0(\alpha)$ and $Y^1(\alpha)$ depends on the relative proportions
of lengths of edges in $Y^0$ and $Y^1$, and not on the value of $\alpha$.
That is,
\begin{align}
\frac{ \norm{A_l}_{Y^0(\alpha)}}{\norm{B_l}_{Y^1(\alpha)}}=\frac{\sqrt{\sum_{e \in A_l} \alpha |e|_{Y^0}}}{\sqrt{\sum_{e \in B_l} \alpha |e|_{Y^1}}} = \frac{ \norm{A_l}_{Y^0}}{\norm{B_l}_{Y^1}}
\end{align}
Now we show that $|e|_{\Gamma_{Y^0(\alpha) Y^1(\alpha)}(t)} = |e|_{\Gamma_{X Y^t}(\alpha)}$.
The combinatorics of the geodesic between $Y^0(\alpha)$ and $Y^1(\alpha)$
do not depend on $\alpha$. Therefore, which edges have positive lengths in the $l^{th}$ leg of $\Gamma_{Y^0(\alpha)Y^1(\alpha)}$ does not depend on $\alpha$.
The length of edge $e$ at $\Gamma_{Y^0(\alpha) Y^1(\alpha)}(t)$ is
\begin{align}
|e|_{\Gamma_{Y^0(\alpha) Y^1(\alpha)}(t)}&=\displaystyle\left\{\begin{array}{ll}
\frac{(1-t)\norm{A_j}_\alpha-t \norm{B_j}_\alpha}{\norm{A_j}_\alpha}|e|_{Y^0(\alpha)}&e\in A_j\\[1em]
\frac{t \norm{B_j}_\alpha-(1-t)\norm{A_j}_\alpha}{\norm{B_j}_\alpha}|e|_{Y^1(\alpha)}&e\in B_j\\[1.5em]
(1-t)|e|_{Y^0(\alpha)}+t |e|_{Y^1(\alpha)}&e\in C
\end{array}.\right.\\
\end{align}
Substituting $\norm{A_j}_\alpha=\alpha \norm{A_j}$, $\norm{B_j}_\alpha=\alpha \norm{B_j}$, \ref{eq:LengthY0}, and \ref{eq:LengthY1} yields
\begin{align}
|e|_{\Gamma_{Y^0(\alpha) Y^1(\alpha)}(t)}&=\displaystyle\left\{\begin{array}{ll}
\alpha \frac{(1-t)\norm{A_j}-t \norm{B_j}}{\norm{A_j}}|e|_{Y^0}&e\in A_j\\[1em]
\alpha \frac{t \norm{B_j}-(1-t)\norm{A_j}}{\norm{B_j}}|e|_{Y^1}&e\in B_j\\[1.5em]
|e|_X + \alpha((1-t)|e|_{Y^0}+t |e|_{Y^1})&e\in C\\
\end{array}.\right.
\end{align}
Now the length of $e$ in $\Gamma_{XY^t}(\alpha)$ is 
\begin{align}\label{eq:LengthYt}
|e|_{\Gamma_{XY^t}(\alpha)}=
\left\{ \begin{array}{ll}
|e|_X + \alpha |e|_{Y^t} & \textrm{if $e \in C$} \\
\alpha |e|_{Y^t} & \textrm{if $e \in E_{Y^t} \setminus C$}\\
\end{array} \right.
\end{align}
The length of $e$ in $Y^t$ is given by 
\begin{align}
|e|_{Y^t}&=\displaystyle\left\{\begin{array}{ll}
\frac{(1-t)\norm{A_j}-t \norm{B_j}}{\norm{A_j}}|e|_{Y^0}&e\in A_j\\[1em]
\frac{t \norm{B_j}-(1-t)\norm{A_j}}{\norm{B_j}}|e|_{Y^1}&e\in B_j\\[1.5em]
((1-t)|e|_{Y^0}+t |e|_{Y^1})&e\in C\\
\end{array}.\right.
\end{align}
Therefore  $|e|_{\Gamma_{Y^0(\alpha) Y^1(\alpha)}(t)} = |e|_{\Gamma_{X Y^t}(\alpha)}$ holds.
\end{proof}

\begin{lem}\label{lem:DirDerConvex}
The directional derivative $F'(X,Y)$ is a convex function of $Y$ over the set of $Y$ such that $\Or(X) \subseteq \Or(Y)$ and $X$ and $Y$ share a vistal facet.
\end{lem}
\begin{proof}
Let $Y^0$ and $Y^1$ be a points in $\T_r$ such that $\Or(X) \subseteq \Or(Y^0)$
and $\Or(X) \subseteq \Or(Y^1)$.
Let $Y^t$ be the point which is proportion $t$ along the geodesic from $Y^0$ to $Y^1$.
Let $\Gamma_{X Y^t}(\alpha):[0,1] \to \T_r$ be a function which parameterizes the geodesic from $X$ to $Y^t$.
Using Lem. \ref{lem:LocalGeoTriangle} and the strict convexity of $F$ together yields
\begin{align}
F(\Gamma_{XY^t}(\alpha))<F(\Gamma_{XY^0}(\alpha))(1-t)+F(\Gamma_{XY^1}(\alpha))t \label{eq:GeodesicConvexity}
\end{align}
The directional derivative from $X$ in the direction of $\Gamma_{XY^t}(\alpha)$ is
\begin{align}
F'(X,Y^t)& = \lim_{\alpha \to 0} \frac{ F(\Gamma_{XY^t}(\alpha))-F(X)}{\alpha}
\end{align}
Substituting for $F(\Gamma_{XY^t}(\alpha))$ using the inequality on line (\ref{eq:GeodesicConvexity}) yields,
\begin{align}
F'(X,Y^t)& \leq \lim_{\alpha \to 0} \frac{ F(\Gamma_{XY^0}(\alpha))(1-t)+F(\Gamma_{XY^1}(\alpha))t-F(X)}{\alpha}
\end{align}
Note that strict inequality may not hold even though the Fr\'{e}chet function is convex because
in the limit the value may approach an infimum. 
Simplifying by separating the fraction and limit reveals that the directional derivative is convex in $Y$,
\begin{align}
F'(X,Y^t) &\leq (1-t)\lim_{\alpha \to 0}\frac{F(\Gamma_{XY^0}(\alpha))-F(X)}{\alpha} + t\lim_{\alpha\to 0}\frac{F(\Gamma_{XY^1}(\alpha))-F(X)}{\alpha}\\
&= (1-t) F'(X,Y^0) + tF'(X,Y^1)
\end{align}
\end{proof}

\begin{lem}\label{lem:DirDerC1}
Let $X$ and $Y$ be points such that $\Or(X)\subseteq \Or(Y)$. $F'(X,Y)$ is a $C^1$ function of $Y$ on the interior of the orthant $\Or(Y)$.
\end{lem}
\begin{proof}
Within any fixed multi-vistal face the algebraic form of $F$ is a sum of smooth functions, and the restricted gradient function is continuous at the boundaries of multi-vistal faces relative to the interior of $\Or(Y)$. 
\end{proof}

Let $X$ and $Y$ be points in $\T_r$ such that $X$ and $Y$ share a pre-multi-vistal facet defined by geodesics from $X$ to $T^1,...,T^n$.
If this is the case, then either (i) $X$ and $Y$ have the same topology, (ii) $X$ is a contraction of $Y$ or (iii) $Y$ is a contraction of $X$.
Assume that if the topologies of trees $X$ and $Y$ differ then $X$ is a contraction of $Y$, that is $\Or(X) \subseteq \Or(Y)$. 
Let $\Gamma(X,Y;\alpha)$,
where $0 \leq \alpha \leq 1$, be the point $\alpha$ proportion along the geodesic from $X$ to $Y$.
In the limit as $\alpha$ approaches 0, the behavior of support pairs from $X$ to $T$ is pivotal in understanding the behavior of the Fr\'{e}chet function on faces of orthants. 
In particular, it is important to distinguish those support pairs which do not exist at a point $X$ on the face of $\Or(Y)$.
\begin{defn}
Let $(A_1,B_1),\ldots,(A_k,B_k)$ be a support sequence for the geodesic from $Y$ to a tree $T$, as in the definition of directional derivative above, Def. \ref{def:DirDer}.
Let any support pair $(A_l,B_l)$ such that $\norm{A_l}_X =0$ be called a \emph{local support pair}.
\end{defn}

Local support pairs will be the earliest support pairs in a support sequence for the geodesic between $Y$ and $T$. 
$Y$ and $X$ share a vistal facet, that is their geodesics to $T$ can be represented with the same support sequence. According to $(P2)$, any support pair such that $\norm{A_l}_X =0$ must be among the first support pairs in the support sequence.
Thus, let $(A_1,B_1),\ldots,(A_m,B_m)$ be local support pairs, and let $(A_{m+1},B_{m+1}),\ldots,(A_k,B_k)$ be the rest of the geodesic support sequence being used to represent the geodesic between $Y$ and $T$.

Let $\tilde{B}$ be all edges from $T$ which are incompatible with at least one edge in $Y$ but not incompatible with any edge in $X$ and let $\tilde{A}$ be all edges from $Y$ which are incompatible
with some edge in $\tilde{B}$. 
\begin{lem}\label{lem:LocalSupportPairsComposition}
Any sequence of local support pairs, $(A_1,B_1),\ldots,(A_m,B_m)$ have the property that the sets $A_1,\ldots,A_m$ partition $\tilde{A}$ and $B_1,\ldots,B_m$ partition $\tilde{B}$.
\end{lem}
\begin{proof}

Any edge in $T$ which is incompatible with an edge in a local support pair, $(A_l,B_l)$ is compatible with every edge in $X$ because for a local support pair $\norm{A_l}_X=0$. Therefore any edge from $T$ in a local support pair must be in $\tilde{B}$. 

An edge in $\tilde{B}$ is compatible with every edge in $X$. Therefore, such an edge cannot be in any of the support pairs with edges from $X$, and thus must be in a local support pair.

Suppose an edge, $e$, from $Y$ is in a local support pair, $(A_l,B_l)$, then it must incompatible with at least one edge in $T$. All the edges which are in $B_l$ must be compatible with all edges in $X$ because  $\norm{A_l}_X=0$. 
Since $e$ is incompatible with some edge in $T$ that is not incompatible with any edge in $X$, $e$ must be in $\tilde{A}$.

Let $e$ be an edge in $\tilde{A}$. Edge $e$ is not in $X$, and edge $e$ is incompatible with at least one edge in $T$ which no edge in $X$ is incompatible with. Edge $e$ must be in a support pair so that along the geodesic the length of $e$ contracts to zero before all the edges in $\tilde{B}$ can switch on. Therefore $e$ must be a support pair with 
at least one of the edges in $\tilde{B}$ that it is incompatible with.
Since all edges in $\tilde{B}$ are in local support pairs, all edges in $\tilde{A}$ must also be in local support pairs. 
\end{proof}
\begin{cor}\label{cor:LocalSupportPairs}
Any sequence of local support pairs which is valid for the geodesic from $Y$ to $T$ is also valid for the geodesic from $Y_\perp$ to $T$ and vice versa. 
\end{cor}
\begin{proof}
Lem. \ref{lem:LocalSupportPairsComposition} implies local support pairs for the geodesic from $Y$ to $T$ and $Y_\perp$ to $T$ would be composed from
the same sets of edges. Factoring out $\alpha$, we see that the relative lengths of edges in $\tilde{A}$ are the same in $Y$ and $Y_\perp$.
\end{proof}

\noindent{\bf Proof of Thm. \ref{thm:DirDerDecomposition}}
\begin{proof}
Note that since $X$ and $Y$ are in the same orthant, the geodesic $\Gamma_{XY}$ is just the line segment $XY$. Let $P=Y-X$, and let $P_X$ and $P_\perp$ be its decomposition into the parts corresponding to $Y_X$ and $Y_\perp$.
Let $Z$ be a point on $XY$ denoted by $|e|_Z = |e|_X + \alpha p_e$.
Let $Z_X=X+\alpha P_X$ and let $Z_\perp=X+\alpha P_\perp$ be the component of $Z$ orthogonal to $\Or(X)$. 
By Cor. \ref{cor:DirDerTangentSpace}, the value of the directional derivative from $X$ to $Y_X$ is
\begin{align}
F'(X,Y_X)=\lim_{\alpha \to 0} \frac{F(Z_X)-F(X)}{\alpha}=\sum_{e\in E_X} p_e \left[\nabla F(X)\right]_e
\end{align}
and the directional derivative from $X$ to $Y_\perp$ is
\begin{align}
F'(X,Y_\perp) &= \lim_{\alpha \to 0} \frac{F(Z_\perp)-F(X)}{\alpha}\\
&=2 \sum_{i=1}^n \left ( \sum_{l: \norm{A^i_l}_X = 0} \left( \norm{B^i_l}\sqrt{\sum_{e \in A^i_l} p_e^2 } \right) -\sum_{e \in C^i \setminus E_X} |e|_{T^i} p_e  \right)
\end{align}

\end{proof}

\noindent{\bf Proof of Thm. \ref{thm:DirDerOfDirDer}}
\begin{proof}
The starting point for proving this claim will be analysis of the difference of function values $F'(X,Y^\alpha)-F'(X,Y)$. 
Lem. \ref{lem:DDWellDefined} states the value of the directional derivative can be expressed by any valid support sequences.
Lem. \ref{lem:DirDerContinuous} states the directional derivative is a continuous function of $Y$.
Taken together these lemmas imply that the values of both $F'(X,Y^\alpha)$ and $F'(X,Y)$ can be expressed
using any valid geodesic support sequences for the geodesic from $T^i$ to $Y^\alpha$, when $\alpha$ is small enough. The difference of function values,  $F'(X,Y^\alpha)-F'(X,Y)$ simplifies to
\begin{align}
F'(X,Y^\alpha)-F'(X,Y)&= \sum_{e\in E_X} \alpha (|e|_{Y'}-|e|_Y)[\nabla F(X)]_e \label{DiffEq1L1}\\
& -\sum_{e \in C^i \setminus E_X} \alpha (|e|_{Y'}-|e|_Y)|e|_{T^i} \label{DiffEq1L2}\\
& + \sum_{i=1}^n \sum_{l:\norm{A^i_l}_X=0} (\norm{A^i_l}_{Y^\alpha}-\norm{A^i_l}_Y)\norm{B^i_l}_{T^i}.\label{DiffEq1L3}
\end{align}
Splitting terms on line (\ref{DiffEq1L2}) into two cases (i) when $ e \in (C^i \setminus E_X) \cap E_Y$ (ii) when $e \in C^i \setminus E_Y$ and splitting terms on line (\ref{DiffEq1L3}) into two cases, (i) when $\norm{A^i_l}_Y>0$ and (ii) when $\norm{A^i_l}_Y=0$ yields 
\begin{align}
F'(X,Y^\alpha)-F'(X,Y)&= \sum_{e\in E_X} \alpha (|e|_{Y'}-|e|_Y)[\nabla F(X)]_e \label{DiffEq2L1}\\
& -\sum_{e \in (C^i \setminus E_X) \cap E_Y} \alpha (|e|_{Y'}-|e|_Y)|e|_{T^i} \label{DiffEq2L2}\\
& -\sum_{e \in C^i \setminus E_Y} \alpha (|e|_{Y'}-|e|_Y)|e|_{T^i} \label{DiffEq2L3}\\
& + \sum_{i=1}^n \sum_{\begin{array}{c} l:\norm{A^i_l}_X=0\\ \;\;\;\;\norm{A^i_l}_Y >0\end{array}} (\norm{A^i_l}_{Y^\alpha}-\norm{A^i_l}_Y)\norm{B^i_l}_{T^i}\label{DiffEq2L4}\\
& + \sum_{i=1}^n \sum_{\begin{array}{c} l:\norm{A^i_l}_X=0\\ \;\;\;\;\norm{A^i_l}_Y =0\end{array}} (\norm{A^i_l}_{Y^\alpha}-\norm{A^i_l}_Y)\norm{B^i_l}_{T^i}.\label{DiffEq2L5}
\end{align}
Now the difference, $F'(X,Y^\alpha)-F'(X,Y)$ is simplified to a state where it is convenient to analyze the derivative, that is,
\begin{align}
\lim_{\alpha \to 0} \frac{F'(X,Y^\alpha)-F'(X,Y)}{\alpha} & = \sum_{e\in E_X}(|e|_{Y'}-|e|_Y)[\nabla F(X)]_e \label{DiffEq3L1}\\
& -\sum_{e \in (C^i \setminus E_X)\cap E_Y}(|e|_{Y'}-|e|_Y)|e|_{T^i} \label{DiffEq3L2}\\
& -\sum_{e \in C^i \setminus E_Y}(|e|_{Y'}-|e|_Y)|e|_{T^i} \label{DiffEq3L3}\\
& + \sum_{i=1}^n \sum_{\begin{array}{c} l:\norm{A^i_l}_X=0\\ \;\;\;\;\norm{A^i_l}_Y >0\end{array}} \sum_{e\in A^i_l} |e|_Y (|e|_{Y'}-|e|_Y) \frac{\norm{B^i_l}_{T^i}}{\norm{A^i_l}_Y}\label{DiffEq3L4}\\
& + \sum_{i=1}^n \sum_{\begin{array}{c} l:\norm{A^i_l}_X=0\\ \;\;\;\;\norm{A^i_l}_Y =0\end{array}} \norm{A^i_l}_{Y'}\norm{B^i_l}_{T^i}.\label{DiffEq3L5}
\end{align}
Taking partial derivatives of $F'(X,Y)$ shows that the sum of terms on lines (\ref{DiffEq3L1}), (\ref{DiffEq3L2}), and (\ref{DiffEq3L4}) is equal to $\sum_{e \in E_Y} (|e|_{Y'}-|e|_Y)[\nabla F'(X,Y)]_e$. One can verify, by using the formula for the value of the directional derivative, Thm. \ref{thm:DirDerValue}, that the sum of terms on lines (\ref{DiffEq3L3}) and (\ref{DiffEq3L5}) is equal to $F'(Y,Y'_{\perp \Or(Y)})$.
\end{proof}

\noindent{\bf Proof of Thm. \ref{thm:RecursiveOptimalityConditions}}
\begin{proof}
If $F'(Y^{k-1},Y^k) \geq 0$ and $\nabla F'(Y^{k-2},Y^{k-1}) \perp \Delta^i$ then $Y^{k-1}$ is a minimizer of $F'(Y^{k-2},Y)$ in $\Or(E)$. By induction, $Y^1$ is a minimizer of $F'(Y^0,Y)$ in $\Or(E)$. Since $F'(Y^0,Y) \geq 0$, $Y^0$ is a minimizer of $F(X)$ in $\Or(E)$.\\
Now the other direction is proved. If $Y^0$ is optimal in $\Or(E)$, then there exists a 
minimizer, $Y^1$ of $F'(Y^0,Y)$ which satisfies the assumptions of the theorem.
Inductively, the minimizer, $Y^{i+1}$, of $F'(Y^i,Y)$ must also satisfy the assumptions of the theorem.
\end{proof}

\section{Concluding remarks and research directions}\label{sec:conclusion}

We obtained relative optimality conditions and algorithms for 
optimizing the Fr\'{e}chet function in an orthant of treespace.
However, a further challenge remains - a method to quickly verify a point is optimal with respect to all orthants which contain it.
The shear number of orthants which contain a point may be
very large with respect to the data, however this is not an indication that there are no polynomial certificates.
In fact, an indicator this problem is not NP-complete is
randomized split-proximal 
point algorithms produce sequences of points with expected distances
to the Fr\'{e}chet mean converging to zero
at a linear rate, and
no approximation methods with such a rate of convergence exists
for NP-complete optimization problems e.g. weighted clique problems. 
One interesting direction is to research why
the maximum clique problem is harder (or not harder) than
finding the orthant of treespace, or equivalently a clique in the split-split compatibility graph (see Sec. \ref{sec:TreeSpace}), which contains the Fr\'{e}chet mean.

Another area of interest is accelerating existing optimization techniques.
One way to accelerate Fr\'{e}chet optimization is to speed up geodesics optimization.
Notice that the systems of equations defining pre-vistal facets and pre-vistal cells
are quadratic cones with cone points at the origin of treespace. 
In squared treespace, the vistal facets and vistal cells are polyhedral cones.
Multi-vistal facets are  also polyhedral cones with cones points
at the origin of treespace because they are
intersections of polyhedral cones share a cone point
at the origin of treespace. 
The nice geometric structure of vistal cells could be useful for determining 
when a search point is on the boundary of a vistal cell, and thus when
the objective function has multiple forms.
Therefore the geometry and combinatorics of vistal cells 
are a mathematical basis for methods to
dynamically update the objective function during line searches.
Such dynamic updating is a topic of further research, and will be the focus of a separate paper.

The necessity to define a fixed set of labeled leaves to use BHV Treespace
limits applications.
The space of treeshapes \cite{Feragen} relaxes the requirement
for labeled leaves and generalizes the attributes of edges from scalars to vectors.
However this flexibility comes at the cost of much more challenging 
optimization problems.
Effective and robust methods for problems, such as computing geodesics and means, are not yet available.

In phylogenetics, when the root of the tree is a common ancestor, there is often a condition that the path in the tree from the root to each other leaf is maintained to a fixed constant.
BHV treespace is also a superset of such trees.
An interesting research direction is optimizing the Fr\'{e}chet mean subject to constraints on the edges of the target tree.

\section*{Acknowledgments}
We would like to acknowledge the assistance of SIOPT Editors and anonymous reviewers for peer review of this article and their advice.

\bibliographystyle{siamplain}
\bibliography{FM_OPT_references}

\end{document}